\documentclass[11pt]{article}
\usepackage{paper,asb}

\usepackage{amsmath}
\usepackage{amsfonts}
\usepackage{amssymb}
\usepackage{amsthm}
\usepackage{hyperref}
\usepackage{comment}
\usepackage[capitalise]{cleveref}
\crefname{subsection}{Subsection}{Subsections}

\newtheorem{theorem}{Theorem}
\newtheorem{lemma}[theorem]{Lemma}
\newtheorem{corollary}[theorem]{Corollary}
\newcommand{\papertitle}{Balancing Communication and Computation in Gradient Tracking Algorithms for Decentralized Optimization}
\newcommand{\paperauthora}{Albert S. Berahas}
\newcommand{\paperauthoraaffiliation}{Department of Industrial and Operations Engineering, University of Michigan}
\newcommand{\paperauthoraemail}{albertberahas@gmail.com}
\newcommand{\paperauthorb}{Raghu Bollapragada}
\newcommand{\paperauthorbaffiliation}{Operations Research and Industrial Engineering Program, University of Texas at Austin}
\newcommand{\paperauthorbemail}{raghu.bollapragada@utexas.edu}
\newcommand{\paperauthorc}{Shagun Gupta}

\newcommand{\paperauthorcemail}{shagungupta@utexas.edu}

\usepackage{booktabs}
\usepackage{caption}
\usepackage{float}
\usepackage{titlesec}
\usepackage{capt-of}

\usepackage{array}
\usepackage{arydshln}
\begin{document}
\title{\papertitle}

\author{\paperauthora\footnotemark[1]
   \and \paperauthorb\footnotemark[2]
   \and \paperauthorc\footnotemark[2]\ \footnotemark[3]}

\maketitle

\renewcommand{\thefootnote}{\fnsymbol{footnote}}
\footnotetext[1]{\paperauthoraaffiliation. (\url{\paperauthoraemail})}
\footnotetext[2]{\paperauthorbaffiliation. (\url{\paperauthorbemail},\url{\paperauthorcemail})}
\footnotetext[3]{Corresponding author.}
\renewcommand{\thefootnote}{\arabic{footnote}}

\begin{abstract}{
Gradient tracking methods have emerged as one of the most popular approaches for solving decentralized optimization problems over networks. In this setting, each node in the network has a portion of the global objective function, and the goal is to collectively optimize this function. At every iteration, gradient tracking methods perform two operations (steps): $(1)$ compute local gradients, and $(2)$ communicate information with local neighbors in the network.
The complexity of these two steps varies across different applications.
In this paper, we present a framework that unifies gradient tracking methods and is endowed with flexibility with respect to the number of communication and computation steps. We establish unified theoretical convergence results for the algorithmic framework with any composition of communication and computation steps, and quantify the improvements achieved as a result of this flexibility. The framework recovers the results of popular gradient tracking methods as special cases, and allows for a direct comparison of these methods. Finally, we illustrate the performance of the proposed methods on quadratic functions and binary classification problems.
}
\end{abstract}



\section{Introduction} \label{sec.intro}

We consider the problem of minimizing a function over a network. In this setting, each node of the network has a portion of the global objective function and the edges represent neighbor nodes that can exchange information, i.e., communicate. The goal is to collectively minimize a finite sum of functions where each component is only known to one of the $n$ nodes (or agents) of the network. Such problems arise in many application areas such as machine learning \cite{forero2010consensus,tsianos2012consensus}, sensor networks \cite{baingana2014proximal, predd2007distributed}, multi-agent coordination \cite{cao2012overview, zhou2011multirobot} and signal processing \cite{combettes2011proximal}. The problem, known as a \emph{decentralized optimization} problem, can be represented as follows:
\begin{align}		\label{eq:prob}
	\min_{x\in \mathbb{R}^d}\quad f(x) = \frac{1}{n} \sum_{i=1}^n f_i(x),
\end{align}
where $f: \mathbb{R}^d \rightarrow \mathbb{R}$ is the global objective function, $f_i: \mathbb{R}^d \rightarrow \mathbb{R}$ for each $i\in \{1,2,...,n \}$ is the local objective function known only to node $i$ and $x\in \mathbb{R}^d$ is the decision variable.

To decouple the computation across different nodes, \eqref{eq:prob} is often reformulated as 
\begin{equation}\label{eq:cons_prob}
\begin{aligned}	
	\min_{x_i \in \mathbb{R}^d}&\quad \frac{1}{n} \sum_{i=1}^n f_i(x_i)\\
    \text{s.t.} &\quad  x_i = x_j, \quad \forall \,\, (i, j) \in \mathcal{E},
\end{aligned}
\end{equation}
where $x_i \in \mathbb{R}^d$ for each node $i\in \{1,2,...,n \}$ is a local copy of the decision variable, and  $\mathcal{E}$ denotes the set of edges of the network; see e.g., \cite{bertsekas2015parallel,nedic2009distributed}. If the underlying network is connected, the \emph{consensus} constraint ensures that all local copies are equal, and, thus, problems \eqref{eq:prob} and \eqref{eq:cons_prob} are equivalent. For compactness, we express problem \eqref{eq:cons_prob} as
\begin{equation}\label{eq:cons_prob1}
\begin{aligned}		
	\min_{x_i \in \mathbb{R}^d}&\quad \textbf{f} (\textbf{x}) = \frac{1}{n} \sum_{i=1}^n f_i(x_i)\\
	\text{s.t.} & \quad (\textbf{W}\otimes I_d)\textbf{x} = \textbf{x}, 
\end{aligned}
\end{equation}
where $\textbf{x} \in \mathbb{R}^{nd}$ is a concatenation of local copies $x_i$, $\textbf{W} \in \mathbb{R}^{n \times n}$ is a matrix that captures the connectivity of the underlying network, $I_d \in \mathbb{R}^{d \times d}$ is the identity matrix of dimension $d$, and the operator $\otimes$ denotes the Kronecker product,  $\textbf{W}\otimes I_d \in \mathbb{R}^{nd \times nd}$. The matrix $\textbf{W}$, known as the \emph{mixing} matrix, is a symmetric, doubly-stochastic matrix with $w_{ii}>0$ and $w_{ij}>0$ ($i\neq j$) if and only if $(i, j) \in \mathcal{E}$ in the underlying network. This matrix ensures that $(\textbf{W}\otimes I_d) \textbf{x}=\textbf{x}$ if and only if $x_i=x_j \,\, \forall \,\, (i, j) \in \mathcal{E}$ in the connected network, thus, 
\eqref{eq:cons_prob} and \eqref{eq:cons_prob1} are equivalent.

In this paper, we focus on gradient tracking methods. These first-order methods update and communicate the local decision variables, and also maintain, update and communicate an additional auxiliary variable that estimates (tracks) the gradient of the global objective function.
We refer to the information shared by the methods as the communication strategy. When applied to the same decentralized setting, the theoretical convergence guarantees and practical implementations of gradient tracking methods with different communication strategies can vary significantly. 
We propose an algorithmic framework that unifies communication strategies in gradient tracking methods and that allows for a direct theoretical and empirical comparison. The framework recovers popular gradient tracking methods as special cases.

The update form of gradient tracking methods can be generalized and decomposed as: $(1)$ one \emph{computation step} of calculating the local gradients, and $(2)$ one \emph{communication step} of sharing information based on the communication strategy. The 
complexity (cost) of these two steps can vary significantly across applications. For example, a large-scale machine learning problem solved on a cluster of computers with shared memory access has a higher cost of computation than communication \cite{tsianos2012consensus}. On the other hand, optimally allocating channels over a wireless sensor network requires economic usage of communications due to limited battery power \cite{magnusson2017bandwidth}.
The subject of developing algorithms (and convergence guarantees) that balance these costs has received significant attention in recent years; see e.g.,~\cite{chen2012fast,berahas2018balancing,9479747,berahas2019nested,sayed2014diffusion,zhang2018communication} and the references therein. In this paper, we follow the approach used in~\cite{berahas2018balancing} and 
explicitly decompose the two steps.
As a result, our algorithms are endowed with flexibility in terms of the number of communication and computation steps performed at each iteration. We show the benefits of this flexibility theoretically and empirically.

\subsection{Literature Review} \label{sec.lit}

Decentralized Gradient Descent (DGD) \cite{bertsekas2015parallel, nedic2009distributed}, a primal first-order method, is considered the prototypical method for solving~\eqref{eq:prob}. 
At each iteration nodes perform local computations and communicate local decision variable to neighbors. 
Gradient tracking methods, e.g., EXTRA \cite{shi2015extra}, SONATA \cite{sun2022distributed}, NEXT \cite{di2016next}, DIGing \cite{nedic2017achieving}, Aug-DGM \cite{xu2015augmented}, have emerged as popular alternatives due to their superior theoretical guarantees and empirical performance. 
They maintain, update and communicate an additional auxiliary variable that tracks the average gradient (additional communication cost compared to DGD). These methods are usually applied to smooth convex functions over undirected networks; however, they are also applicable to various other settings such as time varying networks \cite{nedic2017achieving}, uncoordinated step sizes \cite{nedic2017geometrically, xu2015augmented}, directed networks \cite{nedic2017achieving, pu2020push}, nonconvex functions \cite{di2016next, sun2022distributed}  
and stochastic gradients \cite{pu2021distributed}. Our algorithmic framework generalizes and extends current gradient tracking methodologies, allowing for a unified analysis and direct comparison of popular methods. Notably, our framework differs significantly from existing works that aim to unify gradient tracking methods. In \cite{sundararajan2017robust} and \cite{zhang2019computational}, semi-definite programming is used for this purpose. In \cite{alghunaim2020decentralized} and \cite{xu2021distributed}, the authors introduce unifying frameworks, similar to those proposed in this paper, for comparing different communication strategies. 
However, our framework is simpler and allows for the exact specification of communication and computation steps at each iteration within the network. Furthermore, our proposed framework can accommodate a wider range of communication strategies than those discussed in \cite{sundararajan2017robust,zhang2019computational, alghunaim2020decentralized, xu2021distributed}. As a result of this increased algorithmic flexibility, our framework 
makes it possible to perform comprehensive comparisons among popular gradient tracking methods.


Another class of popular methods is  
primal-dual methods \cite{arjevani2020ideal, jakovetic2014linear, ling2015dlm, shi2014linear, wei20131, mansoori2021flexpd, mancino2021decentralized}. Of these methods, Flex-PD \cite{mansoori2021flexpd} and ADAPD \cite{mancino2021decentralized} allow for flexibility with respect to the number of communication and computation steps. That said, Flex-PD \cite{mansoori2021flexpd} does not show improved performance with the employment of the flexibility and ADAPD \cite{mancino2021decentralized} does not allow for a balance between communication and computation. In \cite{nguyen2022performance}, the authors propose LU-GT, an algorithm that has similarities to our framework in terms of executing multiple local computation steps within gradient tracking methods. Despite the common motivation and similarities, there are several distinct and notable differences. The LU-GT algorithm has two step size hyper-parameters, whereas our approach has only one. 
Furthermore, our analysis results in less pessimistic step size conditions. 
It is worth noting that modifying LU-GT to align with our framework by setting the second step size to one is not possible due to the required conditions imposed in \cite{nguyen2022performance}. 
Moreover, our framework also provides a unifying foundation encompassing all popular gradient tracking methods. 

Finally, algorithms that consider the consensus constraint as a proximal operator have been proposed. These algorithms aim to reduce communication load on distributed systems via a randomization scheme but are primarily designed for fully connected networks (all pairs of nodes are connected). 
Examples of such methods include, but are not limited to, Scaffnew \cite{mishchenko2022proxskip}, FedAvg \cite{li2019convergence}, Scaffold \cite{karimireddy2020scaffold}, Local-SGD \cite{gorbunov2021local} and FedLin \cite{mitra2021linear}.

\subsection{Contributions} \label{sec.contri}
We summarize our main contributions as follows:
\begin{enumerate}
	\item We propose a gradient tracking algorithmic framework (\texttt{GTA}) that unifies communication strategies in gradient tracking methods and provides flexibility in the number of communication and computation steps performed at each iteration. The framework recovers as special cases popular gradient tracking methods, i.e., ~\texttt{GTA-1} \cite{shi2015extra, nedic2017achieving}, \texttt{GTA-2} \cite{di2016next, sun2022distributed} and \texttt{GTA-3} \cite{nedic2017geometrically, xu2015augmented}; see \cref{tab: Algorithm Def}.
    \item We establish the conditions required, on the communication strategy and the step size parameter, that 
    guarantee a global linear rate of convergence for \texttt{GTA} with multiple communication and multiple computation steps. 
    We also compare the relative performance of the special case gradient tracking algorithms, and illustrate the theoretical advantages of \texttt{GTA-3} over \texttt{GTA-2} (and \texttt{GTA-2} over \texttt{GTA-1}), a direct comparison not established in prior literature. 
    \item We show that the rate of convergence improves 
    with increasing the number of communication steps, and the extent of improvement depends on the communication strategy. 
    The improvements are much more profound in \texttt{GTA-3} as compared to \texttt{GTA-2} and \texttt{GTA-1}. 
    \item We illustrate the empirical performance of the proposed \texttt{GTA} framework on quadratic and binary classification logistic regression problems. We show the effect and benefits of 
    multiple communication and/or computation steps per iteration on the performance of the 
    algorithms.
\end{enumerate}

\subsection{Notation} \label{sec.notation}
Our proposed algorithmic framework is iterative and works with inner and outer loops. The variables $x_{i, k, j} \in \mathbb{R}^d$ and $y_{i, k, j} \in \mathbb{R}^d$ denote the local copies 
of the decision variable and the auxiliary variable, respectively, of node $i$, in outer iteration $k$ and inner iteration $j$. The averages of all local decision variables and local auxiliary variables are denoted by $\bar{x}_{k, j} = \frac{1}{n} \sum_{i=1}^n x_{i, k, j}$ and $\bar{y}_{k, j} = \frac{1}{n} \sum_{i=1}^n y_{i, k, j}$, respectively. Boldface lowercase letters represent concatenated vectors of local copies
\begin{align*}
    \xmbf_{k, j} = 
    \begin{bmatrix}
        x_{1, k, j}\\
        x_{2, k, j}\\
        \vdots \\
        x_{n, k, j}
    \end{bmatrix} \in \mathbb{R}^{nd}\mbox{,} \quad
    \ymbf_{k, j} = 
    \begin{bmatrix}
        y_{1, k, j}\\
        y_{2, k, j}\\
        \vdots \\
        y_{n, k, j}
    \end{bmatrix} \in \mathbb{R}^{nd}
    \mbox{,} \quad
      \nabla \fmbf(\xmbf_{k, j}) = 
    \begin{bmatrix}
        \nabla f_1(x_{1, k, j})\\
        \nabla f_2(x_{2, k, j})\\
        \vdots \\
        \nabla f_n(x_{n, k, j})
    \end{bmatrix} \in \mathbb{R}^{nd}.
\end{align*}
The concatenated vector of the average of decision variables ($\Bar{x}_{k, j}$) and auxiliary variables ($\Bar{y}_{k, j}$) repeated $n$ times is denoted by $\xbb_{k, j}$ and $\ybb_{k, j}$, respectively. 
The $n$ dimensional vector of all ones is denoted by $1_n$ and the identity matrix of dimension $n$ is denoted by $I_n$. The spectral radius of square matrix $A$ is $\rho(A)$. Matrix inequalities are defined component wise. 
The Kronecker product of any two matrices $A \in \mathbb{R}^{n \times n}$ and $B \in \mathbb{R}^{d \times d}$ is represented using the operator $\otimes$ and denoted as $A \otimes B \in \mathbb{R}^{nd \times nd}$.

\subsection{Paper Organization} In \cref{sec.methods}, we describe our proposed gradient tracking algorithmic framework (\texttt{GTA}). In \cref{sec.theory}, we provide theoretical convergence guarantees for the proposed algorithmic framework for multiple communication steps and a single computation step at each iteration (\cref{sec.mult comms}) and multiple communication and computation steps at each iteration (\cref{sec.mult grads}). In \cref{sec.full graph res}, we consider the special case 
of fully connected networks. Numerical experiments on quadratic and binary classification logistic regression problems 
are presented in \cref{sec.num_exp}. Finally, we provide concluding remarks in \cref{sec.conc}.

\section{Gradient Tracking Algorithmic Framework}\label{sec.methods}
In this section, we describe our algorithmic framework (\texttt{GTA}) that unifies gradient tracking methods. We then extend the framework to allow for flexibility in the number of communication and  computation steps performed at every iteration. Finally, we make remarks about the algorithmic framework and implementation, and then discuss popular gradient tracking methods as special cases of our proposed framework.

The 
iterate update form (for all $k\geq0$) for the decision variable $\xmbf \in \mathbb{R}^{nd}$ and the auxiliary variable $\ymbf \in \mathbb{R}^{nd}$ that we propose to unify gradient tracking methods is  
\begin{equation}\label{eq: general_form}
\begin{aligned}
    \xmbf_{k+1, 1} & = \Zmbf_1 \xmbf_{k, 1} - \alpha \Zmbf_2 \ymbf_{k,1} \\ 
    \ymbf_{k+1, 1} & = \Zmbf_3 \ymbf_{k, 1} + \Zmbf_4 (\nabla \fmbf(\xmbf_{k+1, 1}) - \nabla \fmbf(\xmbf_{k, 1})), 
\end{aligned}
\end{equation}
where $\alpha>0$ is the constant step size, $\Zmbf_i = \Wmbf_i \otimes I_d \in \mathbb{R}^{nd \times nd}$ for $i = 1, 2, 3, 4$ and $\Wmbf_i \in \mathbb{R}^{n \times n}$ are communication matrices. A communication matrix $\Umbf \in \mathbb{R}^{n \times n}$ is a symmetric, doubly stochastic matrix that respects the connectivity of the network, i.e., $u_{ii} > 0$ and $u_{ij}  \geq  0$ ($i \neq j$) if $(i,j) \in \mathcal{E}$ and $u_{ij} = 0$ ($i \neq j$) if $(i,j) \notin \mathcal{E}$. 
The communication matrices, $\Wmbf_i$ for $i = 1, 2, 3, 4$, represent four (possibly different) network topologies consisting of all the nodes and (possibly different) subsets of the edges of the network over which the corresponding vectors are communicated.
The update form given in \eqref{eq: general_form} generalizes many popular gradient tracking methods for different choices of the communication matrices; see \cref{tab: Algorithm Def}. 
While the  methodology has 
similarities to \cite{xu2021distributed, alghunaim2020decentralized}, our framework allows for the exact specification of the communication quantities within the network and does not 
impose any interdependent conditions among the communication matrices $\Wmbf_i$ for $i = 1, 2, 3, 4$.
In \eqref{eq: general_form} one communication and one computation step is performed at every iteration and so the inner iteration index is always $1$. 
We include this subscript for consistency with the presentation of the algorithm and analysis with multiple communication and computation steps.

\begin{table}[H]\centering
\caption{Special cases of Gradient Tracking Algorithm (\texttt{GTA}).  
}\label{tab: Algorithm Def}
\begin{tabular}{l*{4}{>{\centering\arraybackslash}p{0.8cm}}c}\toprule
\multirow{2}{*}{Method} &\multicolumn{4}{c}{Communication Matrices} & Algorithms in literature\\\cmidrule{2-5}
&$\Wmbf_1$&$\Wmbf_2$&$\Wmbf_3$&$\Wmbf_4$& $(n_c = n_g = 1)$\\\midrule
\texttt{GTA-1} &$\Wmbf$ &$I_n$ &$\Wmbf$ &$I_n$& DIGing \cite{nedic2017achieving}, EXTRA  
\cite{shi2015extra},  \\\hdashline
\texttt{GTA-2} &$\Wmbf$ &$\Wmbf$ &$\Wmbf$ &$I_n$ & SONATA \cite{sun2022distributed}, NEXT \cite{di2016next,pu2020push} \\\hdashline
\texttt{GTA-3} &$\Wmbf$ &$\Wmbf$ &$\Wmbf$ &$\Wmbf$ & Aug-DGM \cite{xu2015augmented}, ATC-DIGing \cite{nedic2017geometrically}\\ 
\bottomrule
\end{tabular}

Note: $\Wmbf$ is a mixing matrix.
\end{table}

We incorporate multiple communications in \eqref{eq: general_form} by replacing $\Zmbf_i$ with $\Zmbf_i^{n_c} = \Wmbf_i^{n_c} \otimes I_d$ for $i=1, 2, 3, 4$, where $n_c \geq 1$ is the number of communication steps at each iteration. 
Taking the communication matrices to the $n_c$ power represents performing $n_c$ communication (consensus) steps at every iteration. We further extend \eqref{eq: general_form} to incorporate multiple computation steps at each iteration. That is, the algorithm performs multiple local updates before communicating information with local neighbors. Our full algorithmic framework with flexibility in the number of communication and computation steps, i.e., $n_c \geq 1$ and $n_g \geq 1$, is given in \cref{alg : Deterministic}. A balance between the number of communication and computation steps is required to achieve overall efficiency for different applications, and \texttt{GTA} allows for such flexibility in these steps via the parameters $n_g$ and $n_c$.

\begin{algorithm}[H]
    \caption{\texttt{GTA}: Gradient Tracking Algorithm}
    \textbf{Inputs:} initial point $\xmbf_{0, 1} \in \R{nd}$, step size $\alpha >0$, computations $n_g \geq 1$, 
    
    communications $n_c \geq 1$.
    \begin{algorithmic}[1]
        \State $\textbf{y}_{0, 1} \gets \nabla \textbf{f}(\textbf{x}_{0, 1})$
        \For{$k \gets 0, 1, 2$ ... }    
            \If{$n_g > 1$}
                \For{$j \gets 1, 2$ ... $, n_g-1$}
                    \State $\textbf{x}_{k, j+1} \gets \textbf{x}_{k, j} - \alpha \,\textbf{y}_{k, j}$
                    \State $\textbf{y}_{k, j+1} \gets \textbf{y}_{k, j} + \nabla \textbf{f}(\textbf{x}_{k, j+1})  - \nabla \textbf{f}(\textbf{x}_{k, j})$
                \EndFor
            \EndIf
            
            \State $\textbf{x}_{k+1, 1} \gets \textbf{Z}_1^{n_c} \textbf{x}_{k, n_g} - \alpha \, \textbf{Z}_2^{n_c} \textbf{y}_{k, n_g}$
            \State $\textbf{y}_{k+1, 1} \gets \textbf{Z}_3^{n_c} \textbf{y}_{k, n_g} + \textbf{Z}_4^{n_c}(\nabla \textbf{f}(\textbf{x}_{k+1, 1})  - \nabla \textbf{f}(\textbf{x}_{k, n_g}))$
        \EndFor
    \end{algorithmic}
    \label{alg : Deterministic}
\end{algorithm}
\bremark 
We make the following remarks about \cref{alg : Deterministic}. 
\begin{itemize}
    \item \textbf{Communications and Computations:} The number of communication and computation steps are dictated by $n_c$ and $n_g$, respectively. 
    By performing multiple communication steps, the goal is to improve consensus across the local decision variables. By performing multiple computation steps, the goal is for individual nodes to make more progress on their local objective functions. 
    \item \textbf{Inner and Outer Loops:} Lines 2--8 form the outer loop and Lines 4--6 form the inner loop. The algorithm performs $n_c$ communication steps each outer iteration (Lines 7 and 8). The algorithm performs $n_g$ local (gradient) computations at each outer iteration; $n_g-1$ computations in the inner loop (Line 6, $\nabla \fmbf(\xmbf_{k, j+1})$) and one computation in the outer loop (Line 8, $\nabla \fmbf(\xmbf_{k+1, 1})$). 
    The inner loop is only executed if more than one computation, i.e., $n_g>1$,  is to be performed every outer iteration (Line 3). By default, we refer to outer iterations when we say iterations unless otherwise specified.
    \item \textbf{Step size ($\alpha>0$):} The algorithm employs a constant step size that depends on the problem parameters, the choices of $n_c$ and $n_g$, and the communication strategy, i.e., $\Wmbf_i$ for $i = 1, 2, 3, 4$. 
\end{itemize} 
\eremark

We analyze \texttt{GTA} and provide results for several popular communication strategies as special cases; summarized in \cref{tab: Algorithm Def}. The choice of the communication matrices ($\Wmbf_i$ for $i = 1, 2, 3, 4$), or equivalently the communication strategy, impact both the convergence of the algorithm and practical implementation. Notice that all methods in \cref{tab: Algorithm Def} require that $\Wmbf_1$ and $\Wmbf_3$ are mixing matrices. Our theoretical results recover this for the general framework. Consider \texttt{GTA-1}, \texttt{GTA-2} and \texttt{GTA-3} defined in \cref{tab: Algorithm Def} with $n_g=1$. In \texttt{GTA-1} and \texttt{GTA-2}, computing local gradients and communications can be performed in parallel because the local gradients need not be communicated ($\Wmbf_4 = I_n$). On the other hand, in \texttt{GTA-3}, these steps need to be performed sequentially. Such trade-offs can create significant impact depending on the problem setting and system.

As mentioned above, the communication matrices ($\Wmbf_i$ for $i = 1, 2, 3, 4$) in \texttt{GTA} need not be the same. That is,  different information can be exchanged on subsets of the edges of the network. This allows for a flexibility in the communication strategy that current gradient tracking methodologies do not possess. 
Such strategies can be useful in applying gradient tracking methods to decentralized settings with networks with bandwidth limitations, e.g., optimization problems in cyberphysical systems with battery powered wireless sensors \cite{magnusson2017bandwidth}.


\section{Convergence Analysis}\label{sec.theory}

In this section, we present 
theoretical convergence guarantees for our proposed algorithmic framework (\texttt{GTA}). 
The analysis is divided into three subsections. 
In \cref{sec.mult comms}, we analyze the effect of multiple communications, i.e., $n_c \geq 1$ (and $n_g = 1$), on \texttt{GTA} and the three special cases \texttt{GTA-1}, \texttt{GTA-2} and \texttt{GTA-3}. 
While these results are a special case of the results presented in \cref{sec.mult grads}, we present these results first as they are simpler to derive, easier to follow and allow us to gain intuition about the effect of the number of communications. 
We then look at the effect of multiple computations in conjunction with multiple communications, i.e., $n_c \geq 1$ and $n_g \geq 1$, in \cref{sec.mult grads} by extending the analysis from \cref{sec.mult comms}. In \cref{sec.full graph res}, we analyze \texttt{GTA-2} and \texttt{GTA-3} for fully connected networks; this special case is not captured by the analysis in the previous subsections. 

We make the following assumption on the functions. 

\bassumption    \label{asum.convex and smooth}
    The global objective function $f: \mathbb{R}^d \rightarrow \mathbb{R}$ is $\mu$-strongly convex. Each component function $f_i: \mathbb{R}^d \rightarrow \mathbb{R}$ $($for $i \in \{ 1,2,\dots,n\}$$)$ has L-Lipschitz continuous gradients. That is, for all $z, z' \in \mathbb{R}^d$
    \begin{align*}
        &f(z')  \geq f(z) +  \langle \nabla f(z), z' - z \rangle + \tfrac{\mu}{2} \|z' - z\|_2^2, \\
        &\|\nabla f_i(z) - \nabla f_i(z')\|_2 \leq L\|z - z'\|_2,    \qquad \qquad \quad \forall \; i = 1, \dots, n.
    \end{align*}
\eassumption
By \cref{asum.convex and smooth}, the global minimizer of~\eqref{eq:prob} is unique, and we denote it by $x^*$.

For notational convenience, we define
\begin{align*}
   \beta^{n_c} = \left\|\Wmbf^{n_c} - \tfrac{1_n1_n^T}{n}\right\|_2, \quad \beta_i^{n_c} = \left\|\Wmbf_i^{n_c} - \tfrac{1_n1_n^T}{n}\right\|_2, \qquad \forall \; i = 1, 2, 3, 4, 
\end{align*}
where $\beta \in [0, 1)$ because $\Wmbf$ is a mixing matrix for a connected network and $\beta_i \in [0,1]$ because $\Wmbf_i$ for $ i = 1, 2, 3, 4$ are symmetric, doubly stochastic matrices. Using the definitions of $\Zmbf^{n_c} = \Wmbf^{n_c}\otimes I_{d}$ and $\Zmbf_i^{n_c} = \Wmbf_i^{n_c}\otimes I_{d}$ for $ i = 1, 2, 3, 4$, it follows that
\begin{align}   \label{eq : beta and Z}
    \|\Zmbf^{n_c} - \Imbf\|_2 = \beta^{n_c}, \quad \|\Zmbf_i^{n_c} - \Imbf\|_2 = \beta_i^{n_c},  \qquad \forall \,\, i = 1, 2, 3, 4.
\end{align}
We also define,
\begin{align}\label{eq : derivative terms define}
   h_{k, j} = \frac{1}{n} \sum_{i = 1}^n \nabla f_i(x_{i, k, j}), \quad \hbar_{k, j} = \frac{1}{n} \sum_{i = 1}^n \nabla f_i(\xbar_{k, j}), \quad \mbox{and} \quad \Imbf = \frac{1_n1_n^T}{n} \otimes I_d .
\end{align}
where $x_{i,k,j}$, denotes the local copy of the $i$th node, at outer iteration $k$ and inner iteration $j$. 
In the analysis, for all $k\geq 0$, we consider the following error vector
\begin{align*} 
    r_k = \begin{bmatrix}
        \|\xbar_{k,1} - x^*\|_2\\
        \|\xmbf_{k,1} - \Bar{\xmbf}_{k,1}\|_2\\
        \|\ymbf_{k,1} - \Bar{\ymbf}_{k,1}\|_2\\
    \end{bmatrix}. 
\end{align*}
The error vector $r_k$ combines the optimization error, $\|\xbar_{k,1} - x^*\|_2$, and consensus errors, $\|\xmbf_{k,1} - \Bar{\xmbf}_{k,1}\|_2$ and $\|\ymbf_{k,1} - \Bar{\ymbf}_{k,1}\|_2$ where $\xmbf_{k, 1}$ and $\ymbf_{k, 1}$ are the first iterates of outer iteration $k$. We establish general technical lemmas that quantify the relation between $r_{k+1}$ and $r_k$ for each case of the presented algorithm.  


\subsection{\texttt{GTA} with multiple communication \texorpdfstring{($n_c \geq 1, n_g=1$)}{Lg}} \label{sec.mult comms} In this section, we analyze \texttt{GTA} when only one computation step is performed per iteration. 
In this setting ($n_g=1$), the inner loop (Lines 4--6 in \cref{alg : Deterministic}) is never executed. Thus, the inner iteration counter in \texttt{GTA} can be ignored and the iteration simplifies to 
\begin{equation}   \label{eq : g=1 general form}
\begin{aligned}
    \xmbf_{k+1} & = \Zmbf_1^{n_c} \xmbf_k - \alpha \Zmbf_2^{n_c} \ymbf_k,   \\ 
    \ymbf_{k+1} & = \Zmbf_3^{n_c} \ymbf_k + \Zmbf_4^{n_c} (\nabla \fmbf(\xmbf_{k+1}) - \nabla \fmbf(\xmbf_{k})).
\end{aligned}
\end{equation}
We note that throughout this subsection we drop the subscript related to the inner iteration $j$, i.e., $x_{i,k,j}$ is denoted as $x_{i,k}$ (since $j=1$), and similar with other quantities. 
We first establish the progression of the error vector $r_k$ as a linear system for \eqref{eq : g=1 general form}. Then, we provide the step size conditions and convergence rates for \eqref{eq : g=1 general form} and the instances in \cref{tab: Algorithm Def} when $n_g = 1$.

\begin{lemma}\label{lem:lyapunov g = 1}
    Suppose \cref{asum.convex and smooth} holds and the number of gradient steps in each outer iteration of \cref{alg : Deterministic} is set to one (i.e., $n_g=1$). If $\alpha \leq \tfrac{1}{L}$, then for all $k\geq 0$,
    \begin{align*}
        r_{k+1} \leq A(n_c) r_k, 
    \end{align*}
    \begin{align} \label{eq : g = 1 general A}
    \mbox{where} \quad A(n_c) = \begin{bmatrix}
        1 - \alpha \mu & \tfrac{\alpha L}{\sqrt{n}} & 0\\
        0 & \beta_1^{n_c} & \alpha\beta_2^{n_c}\\
        \sqrt{n}\alpha \beta_4^{n_c} L^2 & \beta_4^{n_c}L(\|\Zmbf_1^{n_c}-I_{nd}\|_2 + \alpha L) & \beta_3^{n_c} + \alpha \beta_4^{n_c} L\\
        \end{bmatrix}.
    \end{align}
\end{lemma}
\begin{proof}
If $n_g=1$, using \eqref{eq : g=1 general form}, the average iterates can be expressed as
\begin{align*}
    \Bar{x}_{k+1} & = \Bar{x}_k - \alpha \Bar{y}_k,   \\
    \Bar{y}_{k+1} & = \Bar{y}_k + h_{k+1} - h_{k},
\end{align*} 
where $h_k$ is defined in~\eqref{eq : derivative terms define}.
Taking the telescopic sum of $\Bar{y}_{i+1}$ from $i=0$ to $k-1$ with $\bar{y}_0 = h_0$, it follows that
\begin{align*}
    \Bar{y}_{k} & = h_k. \numberthis \label{eq:y_bar_telescope}
\end{align*}

We first consider the optimization error on the average iterates. That is, 
\begin{align*}
    \|\Bar{x}_{k+1} - x^*\|_2 & = \left\|\Bar{x}_k - \alpha \Bar{y}_k + \alpha \hbar_k - \alpha\hbar_k - x^*\right\|_2  \\
    & \leq \left\|\Bar{x}_k- \alpha\hbar_k - x^*\right\|_2 + \alpha \left\|\Bar{y}_k - \hbar_k\right\|_2  \\
    &\leq  (1-\alpha \mu) \|\Bar{x}_k - x^*\|_2 + \alpha \left\|h_k - \hbar_k\right\|_2 \\
    &= (1-\alpha \mu) \|\Bar{x}_k - x^*\|_2 + \tfrac{\alpha}{n}\left\|\sum_{i=1}^n \nabla f_i(x_{i, k}) - \nabla f_i(\Bar{x}_{k})\right\|_2 \\
    & \leq (1-\alpha \mu) \|\Bar{x}_k - x^*\|_2 + \tfrac{\alpha L}{n} \sum_{i=1}^n  \| x_{i, k} - \Bar{x}_{k}\|_2    \\
    & \leq (1-\alpha \mu) \|\Bar{x}_k - x^*\|_2 + \tfrac{\alpha L}{\sqrt{n}}  \| \xmbf_{k} - \Bar{\xmbf}_{k}\|_2    \numberthis \label{eq : g = 1 opt bound}
\end{align*}
where the first inequality is due to the triangle inequality, the second inequality is obtained by performing one gradient descent iteration on function $f$ under \cref{asum.convex and smooth} at the average iterate $\bar{x}_k$ with $\alpha \leq \tfrac{1}{L}$ \cite[Theorem 2.1.14]{nesterov1998introductory} and substituting using \eqref{eq:y_bar_telescope},  the equality is due to \eqref{eq : derivative terms define}, the second to last inequality follows by \cref{asum.convex and smooth}, and the last inequality is due to  $\sum_{i=1}^n \|x_{i, k} - \xbar_k\|_2 \leq \sqrt{n}\|\xmbf_{k} - \Bar{\xmbf}_{k}\|_2$.

Next, we consider the consensus error in $\xmbf_k$,
\begin{align*}
    \xmbf_{k+1} - \Bar{\xmbf}_{k+1} &= \Zmbf_1^{n_c}\xmbf_k - \Bar{\xmbf}_{k} - \alpha \Zmbf_2^{n_c}\ymbf_k + \alpha \Bar{\ymbf}_k \\
    & = \Zmbf_1^{n_c}\xmbf_k - \Zmbf_1^{n_c}\Bar{\xmbf}_{k} - \alpha \Zmbf_2^{n_c}\ymbf_k + \alpha \Zmbf_2^{n_c}\Bar{\ymbf}_k  - \Imbf (\xmbf_k - \Bar{\xmbf}_k) + \Imbf (\ymbf_k - \Bar{\ymbf}_k)\\
    & = \left(\Zmbf_1^{n_c} - \Imbf\right)(\xmbf_k - \Bar{\xmbf}_k) - \alpha\left(\Zmbf_2^{n_c} - \Imbf\right)(\ymbf_k - \Bar{\ymbf}_k). 
\end{align*}
where the second equality follows from adding $- \Imbf (\xmbf_k - \Bar{\xmbf}_k) = 0$ and $\Imbf (\ymbf_k - \Bar{\ymbf}_k) = 0$. By the triangle inequality and~\eqref{eq : beta and Z}, 
\begin{equation}\label{eq : g = 1 x con error}
\begin{aligned}
    \|\xmbf_{k+1} - \Bar{\xmbf}_{k+1}\|_2 & \leq \left\|\Zmbf_1^{n_c} - \Imbf\right\|_2 \|\xmbf_k - \Bar{\xmbf}_k\|_2 + \alpha \left\|\Zmbf_2^{n_c} - \Imbf\right\|_2 \|\ymbf_k - \Bar{\ymbf}_k\|_2   \\ 
    & = \beta_1^{n_c} \|\xmbf_k - \Bar{\xmbf}_k\|_2 + \alpha \beta_2^{n_c} \|\ymbf_k - \Bar{\ymbf}_k\|_2.  
\end{aligned}
\end{equation}

Finally, we consider the consensus error in $\ymbf_k$. By the triangle inequality and~\eqref{eq : beta and Z},
\begin{equation}\label{eq: y_bar_bnd_twoterm}
\begin{aligned}
    &\left\|\ymbf_{k+1} - \Bar{\ymbf}_{k+1}\right\|_2 \\
    =& \left\|\Zmbf_3^{n_c}\ymbf_k - \Bar{\ymbf}_{k} + \Zmbf_4^{n_c} (\nabla \fmbf(\xmbf_{k+1}) - \nabla \fmbf(\xmbf_{k})) - \Imbf (\nabla \fmbf(\xmbf_{k+1}) - \nabla \fmbf(\xmbf_{k}))\right\|_2 \\
    \leq& \left\|\left(\Zmbf_3^{n_c} - \Imbf\right)(\ymbf_k - \Bar{\ymbf}_k)\right\|_2 + \left\|\left(\Zmbf_4^{n_c} - \Imbf\right)(\nabla \fmbf(\xmbf_{k+1}) - \nabla \fmbf(\xmbf_{k}))\right\|_2 \\
    \leq& \beta_3^{n_c}\|\ymbf_k - \Bar{\ymbf}_k\|_2 + \beta_4^{n_c} \left\|\nabla \fmbf(\xmbf_{k+1}) - \nabla \fmbf(\xmbf_{k}) \right\|_2.
\end{aligned}
\end{equation}
The last term in \eqref{eq: y_bar_bnd_twoterm} can be bounded as follows,
\begin{align*}
\left\|\nabla \fmbf(\xmbf_{k+1}) - \nabla \fmbf(\xmbf_{k}) \right\|_2 
&\leq L\|\xmbf_{k+1} - \xmbf_{k}\|_2 \\
&= L\|\Zmbf_1^{n_c}\xmbf_{k} - \alpha \Zmbf_2^{n_c}\ymbf_k - \xmbf_{k}\|_2   \\
&= L\|(\Zmbf_1^{n_c}-I_{nd})(\xmbf_{k} - \Bar{\xmbf}_k) - \alpha \Zmbf_2^{n_c}\ymbf_k\|_2 \\
&\leq L\|\Zmbf_1^{n_c}-I_{nd}\|_2\|\xmbf_{k} - \Bar{\xmbf}_k\|_2 + \alpha L\|\Zmbf_2^{n_c}\|_2\|\ymbf_k + \Bar{\ymbf}_k - \Bar{\ymbf}_k\|_2   \\
&\leq L\|\Zmbf_1^{n_c}-I_{nd}\|_2\|\xmbf_{k} - \Bar{\xmbf}_k\|_2 + \alpha L\|\ymbf_k - \Bar{\ymbf}_k\|_2 + \alpha L\left\|\Bar{\ymbf}_k\right\|_2,\numberthis \label{eq: y_bar_bnd_twoterm33}
\end{align*}
where the first inequality is due to \cref{asum.convex and smooth}, the first equality is due to iterate update form \eqref{eq : g=1 general form}, the second equality is by adding $-(\Zmbf_1^{n_c} - I_{nd})\xbb_{k} = 0$ and the last two inequalities are applications of the triangle inequality.
Next we bound the term $\left\|\Bar{\ymbf}_k\right\|_2$. By \eqref{eq:y_bar_telescope}, \cref{asum.convex and smooth} and $\sum_{i=1}^n \|x_{i, k} - \xbar_k\|_2 \leq \sqrt{n}\|\xmbf_{k} - \Bar{\xmbf}_{k}\|_2$, 
\begin{align*}
\left\|\Bar{\ymbf}_k\right\|_2 & \leq \sqrt{n} \|\bar{y}_k\|_2 \\
&= \sqrt{n} \|h_k\|_2 \\
&\leq \sqrt{n}\left\|\tfrac{1}{n} \sum_{i = 1}^n \nabla f_i(x_{i, k}) -  \tfrac{1}{n} \sum_{i = 1}^n \nabla f_i(\bar{x}_{k})\right\|_2 + \sqrt{n}\left\|\tfrac{1}{n} \sum_{i = 1}^n \nabla f_i(\bar{x}_{k})\right\|_2  \\
&= \tfrac{1}{\sqrt{n}}\left\| \sum_{i = 1}^n \nabla f_i(x_{i, k}) -  \sum_{i = 1}^n \nabla f_i(\bar{x}_{k})\right\|_2 + \tfrac{1}{\sqrt{n}} \left\| \sum_{i = 1}^n \nabla f_i(\bar{x}_{k}) - \sum_{i = 1}^n \nabla f_i(x^*)\right\|_2 \\
&\leq L\left\|\xmbf_k - \bar{\xmbf}_k\right\|_2 + \sqrt{n} L \| \Bar{x}_{k} - x^*\|_2. \numberthis \label{eq : y_bar_bound}
\end{align*}
Thus, by \eqref{eq: y_bar_bnd_twoterm}, \eqref{eq: y_bar_bnd_twoterm33} and \eqref{eq : y_bar_bound}, it follows that
\begin{equation}\label{eq.y_result}
\begin{aligned}
    \left\|\ymbf_{k+1} - \Bar{\ymbf}_{k+1}\right\|_2 
    &\leq \beta_4^{n_c}\sqrt{n}\alpha L^2 \| \Bar{x}_{k} - x^*\|_2 + \beta_4^{n_c}L\left( \|\Zmbf_1^{n_c}-I_{nd}\|_2 + \alpha L\right) \|\xmbf_{k} - \Bar{\xmbf}_k\|_2\\
    & \qquad +  \left( \beta_3^{n_c} + \beta_4^{n_c}\alpha L\right) \|\ymbf_k - \Bar{\ymbf}_k\|_2.
\end{aligned}
\end{equation}
Combining \eqref{eq : g = 1 opt bound}, \eqref{eq : g = 1 x con error} and \eqref{eq.y_result} concludes the proof.
\end{proof}
\qed

Using \cref{lem:lyapunov g = 1}, we now provide the explicit form for $A(n_c)$ in order to establish the progression of the error vector $r_k$ for the special cases defined in \cref{tab: Algorithm Def}.

\begin{corollary} \label{col. A special cases}
Suppose the conditions of \cref{lem:lyapunov g = 1} are satisfied. Then, the matrices $A(n_c)$ for the methods described in \cref{tab: Algorithm Def} are defined as:
\begin{align*}
    \mbox{\texttt{GTA-1}:} & \quad A_1({n_c}) = \begin{bmatrix}
            1 - \alpha \mu & \tfrac{\alpha L}{\sqrt{n}} & 0\\
            0 & \beta^{n_c} & \alpha\\
            \sqrt{n}\alpha L^2 & L(2 + \alpha L) & \beta^{n_c} + \alpha L\\
        \end{bmatrix},\\
    \mbox{\texttt{GTA-2}:} & \quad A_2({n_c})  = \begin{bmatrix}
            1 - \alpha \mu & \tfrac{\alpha L}{\sqrt{n}} & 0\\
            0 & \beta^{n_c} & \alpha \beta^{n_c}\\
            \sqrt{n}\alpha L^2 & L(2 + \alpha L) & \beta^{n_c} + \alpha L\\
        \end{bmatrix},   \numberthis \label{eq : g = 1 algos A}\\
    \mbox{\texttt{GTA-3}:} & \quad A_3({n_c}) = \begin{bmatrix}
            1 - \alpha \mu & \tfrac{\alpha L}{\sqrt{n}} & 0\\
            0 & \beta^{n_c} & \alpha \beta^{n_c}\\
            \beta^{n_c} \sqrt{n}\alpha L^2 & \beta^{n_c}L(2 + \alpha L) & \beta^{n_c}(1 + \alpha L)\\
        \end{bmatrix}.
\end{align*}
\end{corollary}
\begin{proof}
Substituting the matrix values for each method in \eqref{eq : g = 1 general A} and using $\|\Zmbf_1^{n_c} - I_{nd}\|_2 \leq 2$ gives the desired result. 
\end{proof}
\qed
The convergence properties of \texttt{GTA} when $n_g=1$ can be analyzed using the spectral radius of the matrix $A({n_c})$. We now qualitatively establish the effect of $n_c$ on $\rho(A({n_c}))$, the spectral radius of the matrix $A({n_c})$, and the relative ordering between $\rho(A_1({n_c}))$,  $ \rho(A_2({n_c}))$ and $ \rho(A_3({n_c}))$.

\btheorem   \label{th.incr rates}
Suppose \cref{asum.convex and smooth} holds and the number of gradient steps in each outer iteration of  \cref{alg : Deterministic} is set to one (i.e., $n_g=1$). If $\alpha \leq \tfrac{1}{L}$, then as ${n_c}$ increases, $\rho(A({n_c}))$ decreases where $A(n_c)$ is defined in \eqref{eq : g = 1 general A}. Thus, as ${n_c}$ increases, $\rho(A_i({n_c}))$ decreases, for $i=1,2,3$  defined in \eqref{eq : g = 1 algos A}. Moreover, if all three methods described in \cref{tab: Algorithm Def} ({\texttt{GTA-1}}, \texttt{GTA-2} and \texttt{GTA-3}) employ the same step size, 
\begin{align*}
\rho(A_1({n_c})) \geq \rho(A_2({n_c})) \geq \rho(A_3({n_c})),
\end{align*}
where the matrices $A_1({n_c})$, $A_2({n_c})$ and $A_3({n_c})$ are defined in \eqref{eq : g = 1 algos A}.
\etheorem
\bproof
Note that $A(n_c) \geq 0$ and $A(n_c) \geq A(n_c + 1) \geq 0$.  By \cite[Corollary 8.1.19]{horn2012matrix}, it follows that $\rho(A(n_c)) \geq \rho(A(n_c + 1))$. The same argument is applicable for $A_1({n_c})$, $A_2({n_c})$ and $A_3({n_c})$. Now, observe that $A_1({n_c}) \geq A_2({n_c}) \geq A_3({n_c}) \geq 0$ when the same step size is employed. Thus, again by \cite[Corollary 8.1.19]{horn2012matrix}, it follows that $\rho(A_1({n_c})) \geq \rho(A_2({n_c})) \geq \rho(A_3({n_c}))$.
\eproof
\qed

We now derive  conditions for establishing a  linear rate of convergence to the solution for \cref{alg : Deterministic} when $n_g = 1$ in terms of network parameters ($\beta_1, \beta_2, \beta_3, \beta_4$) and objective function parameters ($L$, $\mu$, $\kappa = \tfrac{L}{\mu}$).

\btheorem \label{th. general g=1 step cond}
Suppose \cref{asum.convex and smooth} holds and the number of gradient steps at each outer iteration of \cref{alg : Deterministic} is set to one (i.e., $n_g=1$). If the matrix $A({n_c})$ is irreducible, $\beta_1, \beta_3 < 1$ and 
\begin{align} \label{eq : g = 1 gen step cond}
    \alpha < \min \left\{\tfrac{1}{L}, \tfrac{1 - \beta_3^{n_c}}{L\beta_4^{n_c}} , \tfrac{(1 - \beta_1^{n_c} + 2\beta_2^{n_c})}{2\beta_2^{n_c}\kappa(L + \mu)} \left(\sqrt{1 + \tfrac{4(1-\beta_1^{n_c})(1-\beta_3^{n_c})\beta_2^{n_c}(\kappa+1)}{\beta_4^{n_c}(1 - \beta_1^{n_c} + 2\beta_2^{n_c})^2}} - 1\right)\right\},
\end{align}
then, for all $\epsilon > 0$ there exists a constant $C_\epsilon>0$ such that, for all $k\geq 0$,
\begin{align*}
    \|r_{k}\|_2 \leq C_\epsilon(\rho(A({n_c})) + \epsilon)^k \|r_0\|_2, \quad \text{where } \; \rho(A({n_c})) < 1.
\end{align*}
\etheorem

\bproof
Following \cite[Lemma 5]{pu2021distributed}, derived from the Perron-Forbenius Theorem \cite[Theorem 8.4.4]{horn2012matrix} for a $3\times3$ matrix, when the matrix $A({n_c})$ is non-negative and irreducible, it is sufficient to show that the diagonal elements of $A({n_c})$ are less than one and that $\det(I_3 - A({n_c})) > 0$ in order to guarantee $\rho(A({n_c})) < 1$. We upper bound $\|\Zmbf_1^{n_c} - I_{nd}\|_2 \leq 2$ in $A(n_c)$ for the results.

Let us first consider the diagonal elements of the matrix $A({n_c})$. The first element is, $1 - \alpha \mu \leq 1 -\tfrac{\mu}{L} < 1$ by \eqref{eq : g = 1 gen step cond}. The second element is $\beta_1^{n_c} < 1$ as $\beta_1 < 1$. Finally, the third element is $\beta_3^{n_c} + \alpha\beta_4^{n_c}L < \beta_3^{n_c} + \tfrac{1 - \beta_3^{n_c}}{\beta_4^{n_c}L}\beta_4^{n_c}L = 1$ due to \eqref{eq : g = 1 gen step cond} and $\beta_3 < 1$.

Next, let us consider
\begin{align*}
&\det(I_3 - A({n_c}))\\
= &-\alpha(\alpha^2L^2\beta_2^{n_c}\beta_4^{n_c}\left(L + \mu \right) + \alpha\mu L\beta_4^{n_c}\left(1 - \beta_1^{n_c} + 2\beta_2^{n_c}\right) - \mu\left(1 - \beta_1^{n_c}\right)\left(1 - \beta_3^{n_c}\right)) \\
=&-L^2\beta_2^{n_c}\beta_4^{n_c}(L + \mu)\alpha(\alpha - \alpha_l)(\alpha - \alpha_u), 
\end{align*}
where $\alpha_l = \alpha_1 - \alpha_2$, $\alpha_u = \alpha_1 + \alpha_2$, and 
\begin{align*}
\alpha_1 = \tfrac{-(1 - \beta_1^{n_c} + 2\beta_2^{n_c})}{2\beta_2^{n_c}\kappa(L + \mu)} \quad \text{and} \quad 
\alpha_2 = -\alpha_1\sqrt{1 + \tfrac{4(1-\beta_1^{n_c})(1-\beta_3^{n_c})\beta_2^{n_c}(\kappa+1)}{\beta_4^{n_c}(1 - \beta_1^{n_c} + 2\beta_2^{n_c})^2}}. 
\end{align*}
Observe that $\alpha_l < 0 < \alpha_u$ and $\alpha_2 > |\alpha_1|$.
From \eqref{eq : g = 1 gen step cond}, we have $0<\alpha < \alpha_u$. Therefore, $\det(I_3 - A({n_c})) > 0$, which combined with the fact that the diagonal elements of the matrix are less than 1, implies $\rho(A(n_c)) < 1$.

Finally, we bound the norm of error vector $\|r_k\|_2$ by telescoping $r_{i+1} \leq A(n_c) r_{i}$ from $i = 0$ to $k-1$ and the triangle inequality as
\begin{align*}
    \|r_{k}\|_2 &\leq \|A({n_c})^k\|_2\|r_0\|_2.
\end{align*}
From \cite[Corollary 5.6.13]{horn2012matrix}, we can bound $\|A({n_c})^k\|_2 \leq C_{\epsilon}(\rho(A(n_c)) + \epsilon)^k$ where $\epsilon > 0$ and $C_{\epsilon}$ is a positive constant that depends on $A(n_c)$ and $\epsilon$.
\eproof
\qed

The only constraint \cref{th. general g=1 step cond} imposes on the system (network) is $\beta_1, \beta_3 < 1$. This implies that the communication matrices $\Wmbf_1$ and $\Wmbf_3$ must represent connected networks (not necessarily the same network). Properties of $\Wmbf_2$ and $\Wmbf_4$ change the step size requirement but are not part of the sufficient conditions for convergence. \cref{th. general g=1 step cond} also does not require any relation among $\Wmbf_1$, $\Wmbf_2$, $\Wmbf_3$ and $\Wmbf_4$. This allows for more flexibility than the structures considered in the literature. 
The variables can be communicated along different connections within the network. 
We note that if $A(n_c)$ is a reducible matrix, the analysis for the progression of $r_k$ can be further simplified from \cref{lem:lyapunov g = 1}. For example, when 
$\Wmbf = \tfrac{1_n1_n^T}{n}$, i.e., $\beta = 0$, in \texttt{GTA-2} and \texttt{GTA-3}. 
The analysis for these cases is presented in \cref{sec.full graph res}. 

The next result establishes step size conditions that guarantee a linear rate of convergence for the three special cases (\texttt{GTA-1}, \texttt{GTA-2} and \texttt{GTA-3}).

\bcorollary \label{col. g=1 step cond}
Suppose \cref{asum.convex and smooth} holds, $\Wmbf \neq \tfrac{1_n1_n^T}{n}$, and the number of gradient steps at each outer iteration of \cref{alg : Deterministic} is set to one (i.e., $n_g=1$). If the following step size conditions hold for the methods described in \cref{tab: Algorithm Def},
\begin{align*}
    \mbox{\texttt{GTA-1}:} & \quad \alpha < \min \left\{ \tfrac{1 - \beta^{n_c}}{L} , \tfrac{(3 - \beta^{n_c})}{2\kappa(L + \mu)}\left(\sqrt{1 + 4(\kappa + 1)\left( \tfrac{1 - \beta^{n_c}}{3 - \beta^{n_c}} \right)^2} - 1\right) \right\},\\
    \mbox{\texttt{GTA-2}:} & \quad \alpha < \min \left\{\tfrac{1 - \beta^{n_c}}{L} , \tfrac{(1+\beta^{n_c})}{2\kappa(L + \mu)\beta^{n_c}} \left[ \sqrt{1 + 4(\kappa + 1)\beta^{n_c}\left[\tfrac{1 - \beta^{n_c}}{1+\beta^{n_c}}\right]^2} - 1\right]  \right\}, \\
    \mbox{\texttt{GTA-3}:} & \quad \alpha < \min \left\{\tfrac{1}{L}, \tfrac{1 - \beta^{n_c}}{L\beta^{n_c}} , \tfrac{(1 + \beta^{n_c})}{2\kappa(L + \mu)\beta^{n_c}} \left(\sqrt{1 + 4(\kappa + 1)\left(\tfrac{1 - \beta^{n_c}}{1 + \beta^{n_c}}\right)^2} - 1\right)\right\},
\end{align*}
then, for all $\epsilon > 0$ there exist constants $C_{i,\epsilon} > 0$ such that, for all $k\geq 0$,
\begin{align*}
    \|r_{k}\|_2 \leq C_{i,\epsilon}(\rho(A_i({n_c})) + \epsilon)^k \|r_0\|_2, \; \text{where } \; \rho(A_i({n_c})) < 1, \text{ for } \;  i=1, 2, 3.
\end{align*}
\ecorollary
\bproof
The conditions given in \cref{th. general g=1 step cond} are satisfied for all three methods. That is, the matrices are irreducible as $\Wmbf \neq \tfrac{1_n1_n^T}{n}$, i.e., $\beta > 0$ and $\beta_1, \beta_3 < 1$ in all the three methods as $\beta < 1$ because $\Wmbf$ is mixing matrix of a connected network. Thus, we can use \eqref{eq : g = 1 gen step cond} to derive the conditions on the step size for each of the methods. Substituting the values for $\beta_1,\beta_2,\beta_3$ and $\beta_4$ for each method yields the desired result. We should note that in \texttt{GTA-1} and  \texttt{GTA-2}, we  ignore the term $\tfrac{1}{L}$ since $\tfrac{1}{L} > \tfrac{1-\beta^{n_c}}{L}$. 
\eproof
\qed

\cref{col. g=1 step cond} shows how the communication strategy affects the step size when $n_g = 1$. Among the three methods, \texttt{GTA-3} allows for the largest step size, even having the possibility to use the step size $\tfrac{1}{L}$ if sufficiently large number of communications steps are performed (high $n_c$) and depending on $\beta$. Among \texttt{GTA-1} and \texttt{GTA-2}, \texttt{GTA-2} allows for a larger step size. While these share the same first term in the bound, the presence of the $\beta^{n_c}$ factor in the denominator of the second term in \texttt{GTA-2} makes the bound larger than \texttt{GTA-1}, possibly allowing for a larger step size.

\cref{th. general g=1 step cond} states that there exists a step size such that \texttt{GTA} converges at a linear rate when $n_g = 1$. We now proceed to analyze the convergence rate \texttt{GTA} when $n_g = 1$. Before that, we provide a technical lemma that shows that the largest eigenvalue of the matrix $A(n_c)$ is a positive real number. 

\begin{lemma}\label{lem. g=1 spec norm}
Suppose \cref{asum.convex and smooth} holds, the number of gradient steps at each outer iteration of \cref{alg : Deterministic} is set to one (i.e., $n_g=1$) and $\alpha \leq \tfrac{1}{L}$. If the matrix $A({n_c})$ defined in  \eqref{eq : g = 1 general A} is irreducible, then, the spectral radius of $A({n_c})$ is the largest eigenvalue of $A({n_c})$ and is a positive real number. Consequently, if $\Wmbf \neq \tfrac{1_n1_n^T}{n}$, the spectral radius of matrices $A_1({n_c})$, $A_2({n_c})$, $A_3({n_c})$ defined in \eqref{eq : g = 1 algos A} are also positive real numbers and equal to their largest eigenvalues, respectively.
\end{lemma}

\begin{proof}
The statement about the matrix $A({n_c})$ follows from the Perron-Forbenius Theorem \cite[Theorem 8.4.4]{horn2012matrix}, and the fact that the matrix is 
non-negative and irreducible. Using similar arguments, the statement about the matrices $A_1({n_c})$, $A_2({n_c})$ and $A_3({n_c})$ follows as these matrices are irreducible when $\Wmbf \neq \tfrac{1_n1_n^T}{n}$, i.e., $\beta > 0$.
\end{proof}
\qed

The next theorem provides an upper bound on the convergence rate of \texttt{GTA} for sufficiently small constant step sizes.
\btheorem  \label{th. general g=1 rate bound} 
    Suppose \cref{asum.convex and smooth} holds and the number of gradient steps at each outer iteration of \cref{alg : Deterministic} is set to one (i.e., $n_g=1$). If the matrix $A({n_c})$ is irreducible and $\alpha \leq \tfrac{1}{L}$, then, 
    \begin{align*}
        \rho(A({n_c})) & \leq \,\lambda_u = \max\left\{1 - \tfrac{\alpha\mu}{2}, \hat{\lambda} + \sqrt{2\alpha L \kappa \beta_2^{n_c}\beta_4^{n_c}}\right\}, \numberthis \label{eq : general g = 1 rate}
    \end{align*}
    where $\hat{\lambda} = \tfrac{\beta_1^{n_c} + \beta_3^{n_c} + L\alpha\beta_4^{n_c}  + \sqrt{\left(\beta_1^{n_c} - \beta_3^{n_c} - L\alpha\beta_4^{n_c}\right)^2 + 4\beta_2^{n_c}\beta_4^{n_c} L^2\alpha^2 + 8L\alpha\beta_2^{n_c}\beta_4^{n_c}}}{2}$.
%
\etheorem

\bproof
Using \cref{lem. g=1 spec norm}, we know that the spectral radius of $A(n_c)$ is equal to the largest eigenvalue which is a positive real number.
Following a similar approach to \cite{qu2017harnessing}, we prove $\lambda_u$ is an upper bound on the largest eigenvalue by showing the characteristic equation is non-negative at $\lambda_u$ and strictly increasing for all values greater than $\lambda_u$. Consider 
 \begin{align*}
 g(\lambda) 
= &\det(\lambda I_3 - A({n_c})) \\
= &(\lambda - 1 + \alpha\mu)\left((\lambda - \beta_1^{n_c})(\lambda - \beta_3^{n_c} - \alpha L\beta_4^{n_c}) - \alpha L(2 + \alpha L)\beta_2^{n_c}\beta_4^{n_c}\right) 
\\
& \quad - \alpha^3 L^3 \beta_2^{n_c}\beta_4^{n_c} \\
=& (\lambda - 1 + \alpha\mu) q(\lambda) - \alpha^3 L^3 \beta_2^{n_c}\beta_4^{n_c},
 \end{align*}
where $q(\lambda) = \lambda^2 -\lambda(\beta_1^{n_c} + \beta_3^{n_c} + L\alpha\beta_4^{n_c})  + \beta_1^{n_c}\beta_3^{n_c} + L\alpha\beta_4^{n_c}(\beta_1^{n_c} - 2\beta_2^{n_c} - L\alpha\beta_2^{n_c})$.
Let the roots of the quadratic function $q(\lambda)$ be denoted as $\lambda_1$ and $\lambda_2$. Then, we have,  
\begin{align*}
\max\{\lambda_1, \lambda_2\} 
=& \tfrac{\beta_1^{n_c} + \beta_3^{n_c} + L\alpha\beta_4^{n_c}  + \sqrt{\left(\beta_1^{n_c} + \beta_3^{n_c} + L\alpha\beta_4^{n_c}\right)^2 - 4\left(\beta_1^{n_c}\beta_3^{n_c} + L\alpha\beta_4^{n_c}(\beta_1^{n_c} - 2\beta_2^{n_c} - L\alpha\beta_2^{n_c})\right)}}{2}\\
=& \tfrac{\beta_1^{n_c} + \beta_3^{n_c} + L\alpha\beta_4^{n_c}  + \sqrt{\left(\beta_1^{n_c} - \beta_3^{n_c} - L\alpha\beta_4^{n_c}\right)^2 + 4\beta_2^{n_c}\beta_4^{n_c} L^2\alpha^2 + 8L\alpha\beta_2^{n_c}\beta_4^{n_c}}}{2} 
\end{align*}
Thus, for any 
$\lambda \geq \max\left\{1 - \alpha \mu, \hat{\lambda}\right\}$, the function $g(\lambda)$ is increasing and is lower bounded by $(\lambda - 1 + \alpha\mu)(\lambda - \hat{\lambda})^2 - \alpha^3 L^3 \beta_2^{n_c}\beta_4^{n_c}$. 
By $\lambda_u \geq \max\left\{1 - \alpha \mu, \hat{\lambda}\right\}$, 
\begin{align*}
g(\lambda_u) &\geq (\lambda - 1 + \alpha\mu)(\lambda - \hat{\lambda})^2 - \alpha^3 L^3 \beta_2^{n_c}\beta_4^{n_c} \\
&\geq \left(1 - \tfrac{\alpha\mu}{2} -1 + \alpha\mu\right)(\lambda - \hat{\lambda})^2 - \alpha^3 L^3 \beta_2^{n_c}\beta_4^{n_c} \\
&\geq \tfrac{\alpha \mu}{2}\left(\tfrac{2\alpha L^2\beta_2^{n_c}\beta_4^{n_c}}{\mu}\right)  - \alpha^3 L^3 \beta_2^{n_c}\beta_4^{n_c} \\
&=\alpha^2L^2\beta_2^{n_c}\beta_4^{n_c}(1 - \alpha L) \geq 0,
\end{align*}
where the second and third inequalities are due to the definition of $\lambda_u$ and the final quantity is non-negative since $\alpha \leq \tfrac{1}{L}$. Therefore, by the above arguments, we conclude that $\rho(A({n_c})) \leq \lambda_u$ which completes the proof. 
\eproof
\qed

\cref{th. general g=1 rate bound} is derived independent of the conditions in \cref{th. general g=1 step cond}. 
When $\rho(A(n_c)) < 1$ is imposed using \cref{th. general g=1 rate bound}, $\beta_1,\beta_3 < 1$ is 
a necessary condition for convergence. We show this by constructing a lower bound on $\lambda_u$, $\lambda_u \geq \hat{\lambda} \geq \tfrac{\beta_1^{n_c} + \beta_3^{n_c}+ \left|\beta_1^{n_c} - \beta_3^{n_c}\right|}{2}$.
For convergence we require $\lambda_u < 1$, i.e., $\tfrac{\beta_1 + \beta_3 + |\beta_1 - \beta_3|}{2} < 1$, which implies $\beta_1,\beta_3 < 1$ as $\beta_1, \beta_3 \in [0,1]$. Thus, again we require $\Wmbf_1$ and $\Wmbf_3$ to represent a connected network. The step size condition in \cref{th. general g=1 step cond} is $\mathcal{O}(L^{-1}\kappa^{-0.5})$ while \cref{th. general g=1 rate bound} requires $\mathcal{O}(L^{-1}\kappa^{-1})$, which is more pessimistic. That  said,  the precise and interpretable characterization of the convergence rate in \cref{th. general g=1 rate bound}  allows us to better differentiate amongst the communication strategies and the effect of $n_c$. 

\bcorollary  \label{col. g=1 rate bound}
Suppose \cref{asum.convex and smooth} holds, $\Wmbf \neq \tfrac{1_n1_n^T}{n}$, and the number of gradient steps at each outer iteration of \cref{alg : Deterministic} is set to one (i.e., $n_g=1$). If $\alpha \leq \tfrac{1}{L}$, then, the spectral radii for the methods described in \cref{tab: Algorithm Def} satisfy
\begin{align*}
    \mbox{\texttt{GTA-1}:} & \quad \rho(A_1({n_c}))   \leq \max\left\{1 - \tfrac{\alpha\mu}{2}, \beta^{n_c} + \sqrt{\alpha L} \left(2.5 + \sqrt{2\kappa}\right)\right\},\\
    \mbox{\texttt{GTA-2}:} & \quad \rho(A_2({n_c}))  \leq \max\left\{1 - \tfrac{\alpha\mu}{2}, \beta^{n_c} + \sqrt{\alpha L} \left(2.5 + \sqrt{2\kappa \beta^{n_c}}\right)\right\}, \\
    \mbox{\texttt{GTA-3}:} & \quad \rho(A_3({n_c}))  \leq \max\left\{1 - \tfrac{\alpha\mu}{2}, \beta^{n_c}\left(1 + \sqrt{\alpha L} \left(2.5 + \sqrt{2\kappa}\right)\right)\right\}.
\end{align*}
\ecorollary
\bproof
The conditions in \cref{th. general g=1 rate bound} are satisfied due to \cref{lem. g=1 spec norm}. Thus, we can plug in the values for $\beta_i$ ($i=1,2,3,4$) for each method to get an upper bound on the spectral radii. The upper bound $\lambda_u$ for \texttt{GTA-1} can be simplified as
\begin{align*}
\hat{\lambda} + \sqrt{\tfrac{2\alpha L^2 \beta_2^{n_c}\beta_4^{n_c}}{\mu}} &= \tfrac{2\beta^{n_c} + L\alpha  + \sqrt{5L^2\alpha^2 + 8L\alpha}}{2} + \sqrt{\tfrac{2\alpha L^2}{\mu}} \\
&= \beta^{n_c} + \tfrac{\sqrt{\alpha L}}{2}\left(\sqrt{\alpha L} + 2\sqrt{2\kappa} + \sqrt{8 + 5L\alpha} \right) \\
&\leq \beta^{n_c} + \sqrt{\alpha L} \left(2.5 + \sqrt{2\kappa}\right)
\end{align*}
where the last inequality is due to $\alpha \leq \tfrac{1}{L}$. Following the same approach, $\lambda_u$ for \texttt{GTA-2} can be simplified as
\begin{align*}
\hat{\lambda} + \sqrt{\tfrac{2\alpha L^2 \beta_2^{n_c}\beta_4^{n_c}}{\mu}} &= \tfrac{2\beta^{n_c} + L\alpha  + \sqrt{L^2\alpha^2 + 4L^2\alpha^2\beta^{n_c} + 8L\alpha\beta^{n_c}}}{2} + \sqrt{\tfrac{2\alpha L^2\beta^{n_c}}{\mu}} \\
&= \beta^{n_c} + \tfrac{\sqrt{\alpha L}}{2}\left(\sqrt{\alpha L} + 2\sqrt{2\kappa \beta^{n_c}} + \sqrt{8\beta^{n_c} + 4L\alpha\beta^{n_c} + L\alpha} \right) \\
&\leq \beta^{n_c} + \sqrt{\alpha L} \left(2.5 + \sqrt{2\kappa \beta^{n_c}}\right)
\end{align*}
where the last inequality uses $\alpha \leq \tfrac{1}{L}$ and $\beta < 1$. Finally, the upper bound $\lambda_u$ for \texttt{GTA-3} is
\begin{align*}
\hat{\lambda} + \sqrt{\tfrac{2\alpha L^2 \beta_2^{n_c}\beta_4^{n_c}}{\mu}} &= \tfrac{2\beta^{n_c} + L\alpha\beta^{n_c}  + \sqrt{ 5L^2\alpha^2(\beta^{n_c})^2 + 8L\alpha(\beta^{n_c})^2}}{2} + \sqrt{\tfrac{2\alpha L^2(\beta^{n_c})^2}{\mu}} \\
&= \beta^{n_c}\left(1  + \tfrac{\sqrt{\alpha L}}{2}\left(\sqrt{\alpha L} + 2\sqrt{2\kappa} + \sqrt{8 + 5L\alpha} \right)\right) \\
&\leq \beta^{n_c}\left(1 + \sqrt{\alpha L} \left(2.5 + \sqrt{2\kappa}\right)\right)
\end{align*}
where the last inequality is due to $\alpha \leq \tfrac{1}{L}$ and $\beta < 1$. 
\eproof
\qed

\cref{col. g=1 rate bound} characterizes the effect of  multiple communication steps (when $n_g = 1$) on the convergence rates of \texttt{GTA-1}, \texttt{GTA-2} and \texttt{GTA-3}.
First, the convergence rate improves with increased communications (increase in $n_c$) when $n_g = 1$ for all methods. The improvement is most apparent in \texttt{GTA-3} as increasing $n_c$ drives the second term in the $\max$ bound to zero. Thus, if a sufficient number of communication steps are performed in \texttt{GTA-3}, the method can achieve convergence rates similar to those of gradient descent, i.e., $(1 - \tfrac{\alpha\mu}{2})$. 
The improvement is less apparent in \texttt{GTA-2} and the least in \texttt{GTA-1}. With an increase in $n_c$, the dominating term in the max bound, i.e., $\sqrt{2\alpha L\kappa}$, remains unchanged in \texttt{GTA-1} and changes to $\sqrt{2\alpha L\kappa \beta^{n_c}}$ in \texttt{GTA-2} which is affected by the number of communication steps $n_c$. 


\subsection{\texttt{GTA} with multiple communication and computation \texorpdfstring{($n_c \geq 1, n_g\geq 1$)}{Lg}} \label{sec.mult grads}
In this section, we analyze \texttt{GTA} when multiple computation and communication steps are performed every iteration. We extend the analysis from \cref{sec.mult comms}; the case $n_g = 1$ is a special case of the analysis in this section. The subscript for the inner iteration counter is re-introduced in this section as we consider cases with $n_g > 1$ and the inner loop (Lines 4--6 in \cref{alg : Deterministic}) is executed. We first provide a technical lemma that bounds the errors 
due to the execution of the inner loop. We use this result to extend \cref{lem:lyapunov g = 1} and establish the progression of the error vector $r_k$ when multiple communication and computation steps are performed. Finally, we provide the conditions for linear convergence of \cref{alg : Deterministic} with any composition of communication and computation steps.

\begin{lemma} \label{lem : inner loop deviations}
Suppose \cref{asum.convex and smooth} holds and $\alpha \leq \frac{1}{n_gL}$ in \cref{alg : Deterministic}. Then, for all $k\geq0$ and $1 \leq j \leq n_g$
\begin{align}%
    \|\Bar{\ymbf}_{k,1}\|_2 &\leq \|\ymbf_{k, 1} -  \bar{\ymbf}_{k, 1}\|_2 + L\left\|\xmbf_{k,1} - \bar{\xmbf}_{k,1}\right\|_2 + L\sqrt{n} \| \Bar{x}_{k,1} - x^*\|_2\label{eq:y_bound} \\
    \|\xmbf_{k, j} - \xmbf_{k, 1}\|_2 &\leq 2\alpha (j - 1 ) \|\ymbf_{k, 1}\|_2, \label{eq:iterate_deviation}\\
    \|\xmbf_{k, j} - \xbb_{k, j}\|_2 &\leq 2 \alpha (j - 1)\|\ymbf_{k, 1}\|_2 + \|\xmbf_{k, 1} - \xbb_{k, 1}\|_2, \label{eq:consensus_error_deviation}
\end{align}
\end{lemma}
\bproof
Taking a telescopic sum of $\ymbf_{k, i+1} = \ymbf_{k, i} + \nabla \fmbf(\xmbf_{k, i+1}) - \nabla \fmbf(\xmbf_{k, i})$, the inner loop update, from $i=1$ to $j-1$ we get
\begin{align}
    \ymbf_{k, j} & = \ymbf_{k, 1} + \nabla \fmbf(\xmbf_{k, j}) - \nabla \fmbf(\xmbf_{k, 1}). \label{eq: v telescope}
\end{align}
Using~\eqref{eq: v telescope},  $\ymbf_{k+1, 1}$ can be expressed as
\begin{equation}\label{eq : y general telescope sum}
\begin{aligned}
    \ymbf_{k+1, 1} &= \Zmbf_3^{n_c}\left(\ymbf_{k, 1} + \nabla \fmbf(\xmbf_{k, n_g}) - \nabla \fmbf(\xmbf_{k, 1})\right) + \Zmbf_4^{n_c}\left(\nabla \fmbf(\xmbf_{k+1, 1}) - \nabla \fmbf(\xmbf_{k, n_g})\right)   \\  
    &= \Zmbf_3^{n_c}\ymbf_{k, 1} + \Zmbf_4^{n_c}\nabla \fmbf(\xmbf_{k+1, 1}) - \Zmbf_3^{n_c}\nabla \fmbf(\xmbf_{k, 1}) + \Zmbf_3^{n_c}\nabla \fmbf(\xmbf_{k, n_g}) - \Zmbf_4^{n_c} \nabla \fmbf(\xmbf_{k, n_g})
\end{aligned}
\end{equation}
Taking the component-wise average across all nodes in \eqref{eq: v telescope} and \eqref{eq : y general telescope sum} and using~\eqref{eq : derivative terms define}, it follows that
\begin{align}
    \Bar{y}_{k, j} & = \Bar{y}_{k, 1} + h_{k, j} - h_{k, 1}, \label{eq: g > 1 y_j_telescope}\\
    \Bar{y}_{k+1, 1} & = \Bar{y}_{k, 1} + h_{k+1, 1} - h_{k, 1}.  \label{eq : g > 1 y bar telescope}
\end{align}
Performing a similar telescopic sum as \eqref{eq:y_bar_telescope} with \eqref{eq : g > 1 y bar telescope}, we obtain $\Bar{y}_{k, 1} = h_{k, 1}$. Thus, substituting $\Bar{y}_{k, 1} = h_{k, 1}$ in \eqref{eq: g > 1 y_j_telescope} yields 
\begin{align}   \label{eq: v_bar_telescope}
    \bar{y}_{k, j} &= \bar{y}_{k, 1} + h_{k, j} - h_{k, 1} = h_{k, j}.
\end{align}
By the triangle inequality, $\|\ymbf_{k, 1}\|_2 \leq \|\ymbf_{k, 1} -  \bar{\ymbf}_{k, 1}\|_2   + \| \bar{\ymbf}_{k, 1} \|_2$, where $\|\bar{\ymbf}_{k, 1}\|_2$ can be bounded by a similar procedure to \eqref{eq : y_bar_bound} due to $\Bar{y}_{k, 1} = h_{k, 1}$ to yield \eqref{eq:y_bound}.

Now, taking the telescopic sum of the inner loop update $\xmbf_{k, i} = \xmbf_{k, i-1} - \alpha \ymbf_{k, i-1}$ from $i=2$ to $j$ yields $\xmbf_{k, j} = \xmbf_{k, 1} - \alpha \sum_{i=1}^{j-1}\ymbf_{k, i}$. The sum $\sum_{i = 1}^{j-1} \ymbf_{k, i}$ is evaluated using \eqref{eq: v telescope} as
\begin{equation}\label{eq: sum_y_k_j}
\begin{aligned}   
    \sum_{i = 1}^{j -1} \ymbf_{k, j} & = \ymbf_{k, 1} + \sum_{i = 2}^{j-1} \ymbf_{k, i} + \nabla \fmbf(\xmbf_{k, i}) - \nabla \fmbf(\xmbf_{k, 1}) \\
    &= (j - 1) \ymbf_{k, 1} + \sum_{i = 2}^{j - 1} \nabla \fmbf(\xmbf_{k, i}) - \nabla \fmbf(\xmbf_{k, 1}).
\end{aligned}
\end{equation}
By the triangle inequality and \cref{asum.convex and smooth}, it follows that
\begin{align*}
\|\xmbf_{k, j} - \xmbf_{k, 1}  \|_2 &\leq \alpha(j-1)\|\ymbf_{k, 1}\|_2 + \alpha  \sum_{i=1}^{j-1}\| \nabla \fmbf(\xmbf_{k, i}) - \nabla \fmbf(\xmbf_{k, 1})\|_2 \\
&\leq  \alpha(j-1)\|\ymbf_{k, 1}\|_2 + \alpha L \sum_{i=1}^{j-1}\| \xmbf_{k, i} - \xmbf_{k, 1}\|_2.
\end{align*}
Now we apply induction to show \eqref{eq:iterate_deviation} using the above inequality.
\begin{align*}
    \mbox{For $j = 1$,} \qquad \|\xmbf_{k, 1} - \xmbf_{k, 1}\|_2 &= 0 =  2 \alpha (1 - 1) \|\ymbf_{k, 1}\|_2. \\
    \mbox{For $j \geq 2$,} \qquad \|\xmbf_{k, j} - \xmbf_{k, 1}\|_2 &\leq \alpha(j-1)\|\ymbf_{k, 1}\|_2 + \alpha L \sum_{i=1}^{j-1}\| \xmbf_{k, i} - \xmbf_{k, 1}\|_2 \\
    &\leq \alpha(j-1)\|\ymbf_{k, 1}\|_2 + 2 \alpha^2 L  \sum_{i=1}^{j-1} (i-1) \|\ymbf_{k, 1}\|_2 \\
    &=  \alpha(j-1)\|\ymbf_{k, 1}\|_2 + 2 \alpha^2 L  \|\ymbf_{k, 1}\|_2 \tfrac{(j-2)(j-1)}{2} \\
    &= \alpha(j-1) \left(1 + \alpha L (j-2) \right)\|\ymbf_{k, 1}\|_2  \\
    &\leq 2 \alpha(j-1)\|\ymbf_{k, 1}\|_2,
\end{align*}
where the first equality uses the sum of $j-1$ natural numbers and the second to last inequality is due to $\alpha L \leq \frac{1}{n_g}$ and $j \leq n_g$. 

By \eqref{eq:iterate_deviation}, the triangle inequality and $\|I_{nd} - \Imbf\|_2 = 1$, it follows that
\begin{align*}
    \|\xmbf_{k, j} - \xbb_{k, j}\|_2 &\leq \|\xmbf_{k, j} - \xmbf_{k, 1} + \xbb_{k,1} -  \xbb_{k, j}\|_2 + \|\xmbf_{k, 1} - \xbb_{k, 1}\|_2    \\
    &\leq \|(I_{nd} - \Imbf)(\xmbf_{k, j} - \xmbf_{k, 1})\|_2 + \|\xmbf_{k, 1} - \xbb_{k, 1}\|_2    \\
    &\leq \|\xmbf_{k, j} - \xmbf_{k, 1}\|_2 + \|\xmbf_{k, 1} - \xbb_{k, 1}\|_2.    
\end{align*}
\eproof
\qed

The two bounds in \cref{lem : inner loop deviations} (\eqref{eq:iterate_deviation} and \eqref{eq:consensus_error_deviation}) bound the deviation of the local decision variables from the start of the outer iteration, $\|\xmbf_{k, j} - \xmbf_{k, 1}\|_2$, and the consensus error, $\|\xmbf_{k, j} - \xbb_{k, j}\|_2$,  in inner iteration $j$, respectively. Combined with \eqref{eq:y_bound}, these quantities are bounded as an $\mathcal{O}(\alpha j)$ multiple of the components of the error vector $r_k$. This property has two implications; $(1)$ if one performs more inner iterations, i.e., increases $n_g$, the constant step size $\alpha$ needs to be reduced to reduce these quantities, $(2)$ if an outer iterate is the optimal solution, the inner loop does not introduce any deviations in the iterates and maintains optimality.

We now establish the progression of error vector $r_k$ under multiple communication and computation steps being performed every iteration in \cref{alg : Deterministic}.

\begin{lemma}\label{lem:lyapunov g > 1}
Suppose \cref{asum.convex and smooth} holds and $\alpha \leq \frac{1}{n_g L}$ in \cref{alg : Deterministic}. Then, for all $k\geq 0$,
\begin{align*}
  r_{k+1} \leq B(n_c, n_g) r_k, \quad \text{where $\; B(n_c, n_g) = A(n_c, n_g) + \alpha L (n_g - 1) E(n_c, n_g)$}, \quad     
\end{align*}
\begin{equation}\label{eq : g > 1 general A}
\begin{aligned}
A(n_c, n_g) &= \begin{bmatrix}
    (1 - \alpha \mu)^{n_g}& \frac{\kappa}{\sqrt{n}}(1 - (1 - \alpha\mu)^{n_g}) & 0\\
    0 & \beta_1^{n_c} & \alpha\left((n_g-1)\beta_1^{n_c} + \beta_2^{n_c}\right)\\
    \sqrt{n}\alpha \beta_4^{n_c} L^2 & \beta_4^{n_c}L(\|\Zmbf_1^{n_c}-I_{nd}\|_2 + \alpha L) & \beta_3^{n_c} + \alpha \beta_4^{n_c} L
    \end{bmatrix} ,\\
E(n_c, n_g) &= \begin{bmatrix}
    \alpha L n_g & \frac{\alpha L n_g}{\sqrt{n}} & \frac{\alpha n_g}{\sqrt{n}}\\
    \sqrt{n}\alpha L \delta_1(n_c, n_g) & \alpha L \delta_1(n_c, n_g) & \alpha \delta_1(n_c, n_g)\\
    \sqrt{n} L\delta_2(n_c, n_g) & L \delta_2(n_c, n_g)& \delta_2(n_c, n_g)
    \end{bmatrix},
\end{aligned}
\end{equation}
and
\begin{equation}\label{eq:delta_errors}
\begin{aligned}
    \delta_1(n_c, n_g) &= 2\beta_2^{n_c} + \beta_1^{n_c}(n_g - 2) ,\\
    \delta_2(n_c, n_g) &= 2 \left( \beta_4^{n_c} \|\Zmbf_1^{n_c} - I_{nd}\|_2 + \tfrac{\beta_4^{n_c}}{n_g} + \beta_3^{n_c} \right).
\end{aligned}
\end{equation}
\end{lemma}
\bproof We first consider the optimization error of the average iterates $\xbar_{k, 1}$. Similar to \eqref{eq : g = 1 opt bound}, we  bound the  optimization error as 
\begin{align*}
    \|\Bar{x}_{k, j+1} - x^*\|_2 & \leq (1-\alpha \mu) \|\Bar{x}_{k, j} - x^*\|_2 + \tfrac{\alpha L}{\sqrt{n}} \| \xmbf_{k, j} - \xbb_{k, j}\|_2 \quad  \forall \,\, 1 \leq j \leq n_g - 1,
\end{align*}
where the above holds by using \eqref{eq: v_bar_telescope} (the generalization of \eqref{eq:y_bar_telescope}) and the error bound of gradient descent from \cite[Theorem 2.1.14]{nesterov1998introductory}.  
Next, we bound the optimization error in $\xbar_{k+1, 1}$ with respect to $\xmbf_{k, n_g}$ in a similar manner as, 
\begin{align*}
    \|\Bar{x}_{k+1, 1} - x^*\|_2 & \leq (1-\alpha \mu) \|\Bar{x}_{k, n_g} - x^*\|_2 + \tfrac{\alpha L}{\sqrt{n}} \| \xmbf_{k, n_g} - \xbb_{k, n_g}\|_2.
\end{align*}
Recursively applying the above two bounds, by \eqref{eq:consensus_error_deviation} it follows that,
\begin{align*}
    \|\Bar{x}_{k+1, 1} - x^*\|_2 & 
    \leq (1-\alpha \mu)^{n_g} \|\Bar{x}_{k, 1} - x^*\|_2  + \tfrac{\alpha L}{\sqrt{n}} \sum_{j = 1}^{n_g}(1 - \alpha\mu)^{n_g - j} \|\xmbf_{k,1} - \xbb_{k, 1}\|_2 \\
    & \quad + \tfrac{2\alpha^2 L}{\sqrt{n}} \sum_{j = 1}^{n_g}(1 - \alpha\mu)^{n_g - j} (j-1) \|\ymbf_{k, 1}\|_2\\
    & \leq (1-\alpha \mu)^{n_g} \|\Bar{x}_{k, 1} - x^*\|_2  + \tfrac{\kappa}{\sqrt{n}} \left[1 - (1 - \alpha\mu)^{n_g}\right] \|\xmbf_{k, 1} - \xbb_{k, 1}\|_2 \\
    & \quad + \tfrac{\alpha^2 L}{\sqrt{n}} n_g(n_g - 1) \|\ymbf_{k, 1}\|_2,
\end{align*}
where the last inequality is due to the fact that $(1 - \alpha\mu)^{n_g - j} \leq 1 \,\forall j = 1, 2, ..., n_g$ due to $\alpha \leq \frac{1}{L n_g}$, the coefficient of the second term is the sum of a geometric progression, and the coefficient of the third term is the sum of the first $n_g-1$ natural numbers. 
By \eqref{eq:y_bound}, we obtain the desired bound on the optimization error.

Next, we consider the consensus error in $\xmbf_{k, 1}$,
\begin{align*}
    &\xmbf_{k + 1, 1} - \xbb_{k+1, 1} \\
    =& \left(I_{nd} - \Imbf\right)\xmbf_{k+1, 1} = \left(I_{nd} - \Imbf \right)(\Zmbf_1^{n_c} \xmbf_{k, n_g} - \alpha \Zmbf_2^{n_c}\ymbf_{k, n_g})    \\
    =& \left(I_{nd} - \Imbf \right)\left(\Zmbf_1^{n_c}\left(\xmbf_{k, 1} - \alpha\sum_{j = 1}^{n_g - 1}\ymbf_{k, j}\right) - \alpha \Zmbf_2^{n_c}\ymbf_{k, n_g}\right) \\
    =& \left(I_{nd} - \Imbf\right)\left(\Zmbf_1^{n_c}\xmbf_{k, 1} - \alpha \Zmbf_1^{n_c}\left((n_g-1) \ymbf_{k, 1} + \sum_{j = 2}^{n_g-1} \nabla \fmbf(\xmbf_{k, j}) - \nabla \fmbf(\xmbf_{k, 1})\right)\right) \\
    & \; - \alpha \left(I_{nd} - \Imbf\right) \left(\Zmbf_2^{n_c}( \ymbf_{k, 1} + \nabla \fmbf(\xmbf_{k, n_g}) - \nabla \fmbf(\xmbf_{k, 1}))\right)\\
    =& \left(\Zmbf_1^{n_c} - \Imbf\right)(\xmbf_{k, 1} - \xbb_{k, 1}) -\alpha \left((n_g-1)\left(\Zmbf_1^{n_c} - \Imbf\right) + \left(\Zmbf_2^{n_c} - \Imbf\right)\right) (\ymbf_{k, 1} - \ybb_{k, 1})  \\
    & \; - \alpha\left(\Zmbf_2^{n_c} - \Imbf\right)\left(\nabla \fmbf(\xmbf_{k, n_g}) - \nabla \fmbf(\xmbf_{k, 1})\right) - \alpha\left(\Zmbf_1^{n_c} - \Imbf\right)\left(\sum_{j = 2}^{n_g-1} \nabla \fmbf(\xmbf_{k, j}) - \nabla \fmbf(\xmbf_{k, 1})\right) 
\end{align*}
where the second equality is a telescopic sum of the inner loop update ($\xmbf_{k, i} = \xmbf_{k, i-1} - \alpha \ymbf_{k, i-1}$) from $i=2$ to $n_g$ and the third equality is due to \eqref{eq: sum_y_k_j}. By the triangle inequality, \cref{asum.convex and smooth} and \eqref{eq : beta and Z},
\begin{align*}
    \|\xmbf_{k + 1, 1} - \xbb_{k+1, 1}\|_2 &\leq \beta_1^{n_c}\|\xmbf_{k, 1} - \xbb_{k, 1}\|_2 + \alpha \left((n_g-1)\beta_1^{n_c} + \beta_2^{n_c}\right) \|\ymbf_{k, 1} - \ybb_{k, 1}\|_2 \\
    & \quad + \alpha\beta_2^{n_c} L\|\xmbf_{k, n_g} - \xmbf_{k, 1}\|_2 + \alpha\beta_1^{n_c} L \sum_{j = 2}^{n_g-1} \|\xmbf_{k, j} - \xmbf_{k, 1}\|_2.
\end{align*}
Adding $\alpha\beta_1^{n_c} L \|\xmbf_{k, 1} - \xmbf_{k, 1}\|_2 = 0$ to the right hand side and \eqref{eq:iterate_deviation}, it follows,
\begin{align*}
    \|\xmbf_{k + 1, 1} - \xbb_{k+1, 1}\|_2 &\leq \beta_1^{n_c}\|\xmbf_{k, 1} - \xbb_{k, 1}\|_2 + \alpha \left((n_g-1)\beta_1^{n_c} + \beta_2^{n_c}\right) \|\ymbf_{k, 1} - \ybb_{k, 1}\|_2 \\
    & \quad + \alpha^2\beta_2^{n_c} L (2(n_g - 1)) \|y_{k, 1}\|_2 + \alpha^2\beta_1^{n_c} L \sum_{j = 1}^{n_g-1} 2(j - 1) \|y_{k, 1}\|_2 \\
    &= \beta_1^{n_c}\|\xmbf_{k, 1} - \xbb_{k, 1}\|_2 + \alpha \left((n_g-1)\beta_1^{n_c} + \beta_2^{n_c}\right) \|\ymbf_{k, 1} - \ybb_{k, 1}\|_2 \\
    & \quad + 2\alpha^2 L (n_g - 1) \left( \beta_2^{n_c} + \beta_1^{n_c}\tfrac{(n_g - 2)}{2}  \right)\|\ymbf_{k, 1}\|_2.
\end{align*}
The desired bound for the consensus error in $\xmbf_{k, 1}$ follows by using \eqref{eq:y_bound}.

Finally, we consider the consensus error in $\ymbf_{k, 1}$. By \eqref{eq : y general telescope sum},
\begin{align*}
    \ymbf_{k + 1, 1} - \ybb_{k+1, 1} &= \left(I_{nd} - \Imbf\right)\ymbf_{k+1, 1} \\
    &= \left(I_{nd} - \Imbf\right)(\Zmbf_3^{n_c}\ymbf_{k, 1} + \Zmbf_4^{n_c}\nabla \fmbf(\xmbf_{k+1, 1}) - \Zmbf_3^{n_c}\nabla \fmbf(\xmbf_{k, 1})) \\
    &\quad + \left(I_{nd} - \Imbf\right)(\Zmbf_3^{n_c}\nabla \fmbf(\xmbf_{k, n_g}) - \Zmbf_4^{n_c} \nabla \fmbf(\xmbf_{k, n_g}))\\
    & = \left(\Zmbf_3^{n_c} - \Imbf\right)(\ymbf_{k, 1} - \ybb_{k, 1}) + \left(\Zmbf_4^{n_c} - \Imbf\right)(\nabla \fmbf(\xmbf_{k+1, 1}) - \nabla \fmbf(\xmbf_{k, n_g})) \\
    & \quad + \left(\Zmbf_3^{n_c} - \Imbf\right)(\nabla \fmbf(\xmbf_{k, n_g}) - \nabla \fmbf(\xmbf_{k, 1}) )
\end{align*}
By \cref{asum.convex and smooth} and \eqref{eq : beta and Z},
\begin{align*}
    &\|\ymbf_{k + 1, 1} - \ybb_{k+1, 1}\|_2 \\
    &\leq \beta_3^{n_c}\|\ymbf_{k, 1} - \ybb_{k, 1}\|_2 + \beta_4^{n_c} L\|\xmbf_{k+1, 1} - \xmbf_{k, n_g}\|_2 + \beta_3^{n_c} L\|\xmbf_{k, n_g} - \xmbf_{k, 1}\|_2 \\
    &= \beta_3^{n_c}\|\ymbf_{k, 1} - \ybb_{k, 1}\|_2 + \beta_4^{n_c} L\|(\Zmbf_1^{n_c} - I_{nd})(\xmbf_{k, n_g} - \xbb_{k, n_g}) - \alpha \Zmbf_2^{n_c} \ymbf_{k, n_g}\|_2   \\
    & \quad+ \beta_3^{n_c} L\|\xmbf_{k, n_g} - \xmbf_{k, 1}\|_2\\
    &\leq \beta_3^{n_c}\|\ymbf_{k, 1} - \ybb_{k, 1}\|_2 + \beta_4^{n_c} L\|\Zmbf_1^{n_c} - I_{nd}\|_2\|\xmbf_{k, n_g} - \xbb_{k, n_g}\|_2 + \alpha \beta_4^{n_c} L \|\Zmbf_2^{n_c}\|_2\| \ymbf_{k, n_g}\|_2 \\
    & \quad  + \beta_3^{n_c} L\|\xmbf_{k, n_g} - \xmbf_{k, 1}\|_2 \\
    &= \beta_3^{n_c}\|\ymbf_{k, 1} - \ybb_{k, 1}\|_2 + \beta_4^{n_c} L\|\Zmbf_1^{n_c} - I_{nd}\|_2\|\xmbf_{k, n_g} - \xbb_{k, n_g}\|_2 \\
    & \quad + \alpha \beta_4^{n_c} L \| \ymbf_{k, 1} + \nabla \fmbf(\xmbf_{k, n_g}) - \nabla \fmbf(\xmbf_{k, 1})\|_2 + \beta_3^{n_c} L\|\xmbf_{k, n_g} - \xmbf_{k, 1}\|_2 \\
    &\leq \beta_3^{n_c}\|\ymbf_{k, 1} - \ybb_{k, 1}\|_2 + \beta_4^{n_c} L\|\Zmbf_1^{n_c} - I_{nd}\|_2\|\xmbf_{k, n_g} - \xbb_{k, n_g}\|_2 \\
    & \quad + \alpha \beta_4^{n_c} L \| \ymbf_{k, 1}\|_2 + \alpha \beta_4^{n_c} L^2 \|\xmbf_{k, n_g} - \xmbf_{k, 1}\|_2 + \beta_3^{n_c} L\|\xmbf_{k, n_g} - \xmbf_{k, 1}\|_2
\end{align*}
where the first equality follows from $\xmbf_{k+1, 1} = \Zmbf_1^{n_c} \xmbf_{k, n_g} - \alpha\Zmbf_2^{n_c}\ymbf_{k, n_g}$ and $- (\Zmbf_1^{n_c} - I_{nd})\xbb_{k, n_g} = 0$, the second inequality is by the triangle inequality, the second equality follows by 
\eqref{eq: v telescope}, and the last inequality is an application of triangle inequality and \cref{asum.convex and smooth}. By \eqref{eq:iterate_deviation}, \eqref{eq:consensus_error_deviation} and $\alpha L \leq \frac{1}{n_g}$, it follows,
\begin{align*}
    &\|\ymbf_{k + 1, 1} - \ybb_{k+1, 1}\|_2 \\
    \leq &\beta_3^{n_c}\|\ymbf_{k, 1} - \ybb_{k, 1}\|_2 + \beta_4^{n_c} L\|\Zmbf_1^{n_c} - I_{nd}\|_2 \|\xmbf_{k, 1} - \xbb_{k, 1}\|_2 \\
    &  + \left( \alpha \beta_4^{n_c} L + 2\alpha (n_g-1) L \left( \beta_4^{n_c} \|\Zmbf_1^{n_c} - I_{nd}\|_2 + \tfrac{\beta_4^{n_c}}{n_g}  + \beta_3^{n_c} \right) \right) \|\ymbf_{k, 1}\|_2.
\end{align*}
Substituting \eqref{eq:y_bound} yields the desired bound for the consensus error in $\ymbf_{k,1}$.
\eproof
\qed

\cref{lem:lyapunov g > 1} quantifies the progression of error vector $r_k$ using the matrix $B(n_c, n_g)$, similar to \cref{lem:lyapunov g = 1} but now allowing for multiple computation steps. Notice that when $n_g = 1$, \cref{lem:lyapunov g > 1} reduces to \cref{lem:lyapunov g = 1}, making it a special case of this analysis. We split the matrix $B(n_c, n_g)$ into the matrices $A(n_c, n_g)$ and $E(n_c, n_g)$. The latter matrix is characterized by the terms $\delta_1(n_c, n_g)$ and $\delta_2(n_c, n_g)$. 
We now define the explicit form of $B(n_c, n_g)$ 
for the methods defined in \cref{tab: Algorithm Def}.

\bcorollary \label{col. B special cases}
Suppose the conditions of \cref{lem:lyapunov g > 1} are satisfied. Then, the matrices $A(n_c, n_g)$ for the methods described in \cref{tab: Algorithm Def} are defined as:
\begin{equation} \label{eq : g > 1 algos A}
    \begin{aligned}
        \mbox{\texttt{GTA-1}:} & \quad A_1({n_c, n_g}) = \begin{bmatrix}
        (1 - \alpha \mu)^{n_g}& \frac{\kappa}{\sqrt{n}}(1 - (1 - \alpha\mu)^{n_g}) & 0\\
        0 & \beta^{n_c} & \alpha\left((n_g-1)\beta^{n_c} + 1\right)\\
        \sqrt{n}\alpha L^2 & L(2 + \alpha L) & \beta^{n_c} + \alpha L
    \end{bmatrix},\\
    \mbox{\texttt{GTA-2}:} & \quad A_2({n_c, n_g})  = \begin{bmatrix}
        (1 - \alpha \mu)^{n_g}& \frac{\kappa}{\sqrt{n}}(1 - (1 - \alpha\mu)^{n_g}) & 0\\
        0 & \beta^{n_c} & \alpha\beta^{n_c}n_g\\
        \sqrt{n}\alpha L^2 & L(2 + \alpha L) & \beta^{n_c} + \alpha  L
    \end{bmatrix}, \\
    \mbox{\texttt{GTA-3}:} & \quad A_3({n_c, n_g}) = \begin{bmatrix}
        (1 - \alpha \mu)^{n_g}& \frac{\kappa}{\sqrt{n}}(1 - (1 - \alpha\mu)^{n_g}) & 0\\
        0 & \beta^{n_c} & \alpha\beta^{n_c}n_g\\
        \sqrt{n}\alpha \beta^{n_c} L^2 & \beta^{n_c}L(2 + \alpha L) & \beta^{n_c}(1 + \alpha L)
    \end{bmatrix}.
    \end{aligned}
\end{equation}
The matrix $E(n_c, n_g)$ for the methods described in \cref{tab: Algorithm Def} is defined using the error terms $(\delta_1(n_c, n_g)$ and $\delta_2(n_c, n_g))$. 
The error terms for the methods described in \cref{tab: Algorithm Def} are defined in \cref{tab: n_g > 1 error terms special cases}.
\begin{table}[H]\centering
\caption{Error terms ($\delta_1(n_c, n_g)$ and $\delta_2(n_c, n_g)$) for \texttt{GTA-1}, \texttt{GTA-2} and \texttt{GTA-3}}\label{tab: n_g > 1 error terms special cases}
\begin{tabular}{lcc}\toprule
Method & $\delta_1(n_c, n_g)$ & $\delta_2(n_c, n_g)$ \vspace{2pt} \\ \hline \\[-8pt]
\texttt{GTA-1} & $2 + \beta^{n_c}(n_g - 2)$ & $2 \left( 2 + \tfrac{1}{n_g} + \beta^{n_c} \right)$ \vspace{2pt} \\ \hdashline \\[-10pt]
\texttt{GTA-2} & $n_g\beta^{n_c}$ & $2 \left( 2 + \tfrac{1}{n_g} + \beta^{n_c} \right)$ \vspace{2pt} \\ \hdashline \\[-10pt]
\texttt{GTA-3} & $ n_g\beta^{n_c} $ & $2 \beta^{n_c} \left( 3 + \tfrac{1}{n_g} \right)$ \vspace{2pt} \\
\bottomrule
\end{tabular}
\end{table}
\ecorollary
\bproof
Substituting the matrix values for each method in \eqref{eq : g > 1 general A} and bounding $\|\Zmbf_1^{n_c} - I_{nd}\|_2 \leq 2$ gives the desired result.
\eproof
\qed

\cref{col. B special cases} presents the explicit form of the matrices $B_i(n_c, n_g) = A_i(n_c, n_g) + \alpha L (n_g - 1) E_i(n_c, n_g)$ for $i = 1, 2, 3$, for each of the methods in \cref{tab: Algorithm Def}. The convergence properties of \texttt{GTA} can be analyzed using the spectral radius of $B(n_c, n_g)$. We now qualitatively establish the effect of the number of communication steps $n_c$ on $\rho(B(n_c, n_g))$ and a relative ordering for $\rho(B_1(n_c, n_g))$, $\rho(B_2(n_c, n_g))$ and $\rho(B_3(n_c, n_g))$.

\btheorem   \label{th.incr rates g > 1}
Suppose \cref{asum.convex and smooth} holds. If $\alpha \leq \frac{1}{L n_g}$ in \cref{alg : Deterministic}, then
as $n_c$ increases, $\rho(B(n_c, n_g))$ decreases where $B(n_c, n_g)$ is defined in \cref{lem:lyapunov g > 1}. Thus, as $n_c$ increases, $\rho(B_i(n_c, n_g))$ decreases for all $i =1, 2, 3$ defined in \cref{col. B special cases}. Moreover, if all three methods defined in \cref{tab: Algorithm Def} (\texttt{GTA-1}, \texttt{GTA-2} and \texttt{GTA-3}) employ the same step size,
\begin{align*}
\rho(B_1({n_c, n_g})) \geq \rho(B_2({n_c, n_g})) \geq \rho(B_3({n_c, n_g})).
\end{align*}
\etheorem

\bproof
Note that $A(n_c, n_g) \geq 0$ and $E(n_c, n_g) \geq 0$, thus $B(n_c, n_g) \geq 0$. Also, $A(n_c, n_g) \geq A(n_c + 1, n_g) $, $\delta_1(n_c, n_g) \geq \delta_1(n_c + 1, n_g)$, $\delta_2(n_c, n_g) \geq \delta_2(n_c + 1, n_g)$, thus $E(n_c, n_g) \geq E(n_c + 1, n_g)$ and $B(n_c, n_g) \geq B(n_c + 1, n_g)$.  By \cite[Corollary 8.1.19]{horn2012matrix}, it follows that $\rho(A(n_c, n_g)) \geq \rho(A(n_c + 1, n_g))$, $\rho(E(n_c, n_g)) \geq \rho(E(n_c + 1, n_g))$ and $\rho(B(n_c, n_g)) \geq \rho(B(n_c + 1, n_g))$. The same argument is applicable for $B_1({n_c, n_g})$, $B_2({n_c, n_g})$ and $B_3({n_c, n_g})$. Now, observe that $B_1({n_c, n_g}) \geq B_2({n_c, n_g}) \geq B_3({n_c, n_g}) \geq 0$ when the same step size is employed. Thus, again by \cite[Corollary 8.1.19]{horn2012matrix}, it follows that $\rho(B_1({n_c, n_g})) \geq \rho(B_2({n_c, n_g})) \geq \rho(B_3({n_c, n_g}))$.
\eproof
\qed

The effect of the number of computation steps $n_g$ on $\rho(B({n_c, n_g}))$ is not as clear as the effect of the number of communication steps $n_c$. Increasing $n_g$ increases all elements of the matrix $\alpha L (n_g - 1)E(n_c, n_g)$, while $(1 - \alpha \mu)^{n_g}$ in the matrix $A(n_c, n_g)$ decreases since $\alpha \leq \frac{1}{Ln_g}$. Thus, the effect of $n_g$ on $\rho(B(n_c, n_g))$ is not monotonic. 

We now derive conditions for establishing a 
linear rate of convergence for \cref{alg : Deterministic} with multiple communication and computation steps every iteration in terms of network parameters ($\beta_1, \beta_2, \beta_3, \beta_4$) and objective function parameters ($L, \mu, \kappa = \frac{L}{\mu}$).

\btheorem  \label{th. alpha bound g > 1}
Suppose \cref{asum.convex and smooth} holds and a finite number of computation steps are performed at each outer iteration of \cref{alg : Deterministic} (i.e., $1 \leq n_g < \infty$). If the matrix $B(n_c, n_g)$ is irreducible, $\beta_1, \beta_3 < 1$ and
\begin{align} \label{eq : alpha mult grads gen}
    \alpha < \min \left\{\tfrac{1}{n_g L}, \tfrac{\mu}{(2L^2 + \mu^2)(n_g - 1)}, \tfrac{1}{2L} \sqrt{\tfrac{3(1 - \beta_1^{n_c})}{\delta_1(n_c, n_g) (n_g - 1)}}, \tfrac{3(1 - \beta_3^{n_c})}{4L(\beta_4^{n_c} + \delta_2(n_c, n_g)(n_g - 1))}, \tfrac{ - b_2 + \sqrt{b_2^2 + 4b_1b_3}}{2b_1}\right\}
\end{align}
where
\begin{align*}
    b_1 =& \tfrac{\mu L^2 n_g}{2} \left[(n_g-1) \left( \beta_1^{n_c} + \delta_1(n_c, n_g)\right) + \beta_2^{n_c} \right]\left[\beta_4^{n_c} +  (n_g - 1)\delta_2(n_c, n_g)\right] \\
    &+L^3 n_g (n_g - 1) \left[ \delta_1(n_c, n_g)  \left(\tfrac{1 - \beta_3^{n_c}}{4}\right) + (\beta_4^{n_c} + (n_g - 1)\delta_2(n_c, n_g)) \left(\tfrac{1 - \beta_1^{n_c}}{4}\right) \right]\\
    &+ L^2 (n_g - 1)^2 \left[ L\delta_1(n_c, n_g)\left(3\beta_4^{n_c} + (n_g - 1)\delta_2(n_c, n_g)\right) +  \delta_1(n_c, n_g) \left(\tfrac{1 - \beta_3^{n_c}}{4}\right)\right] \\
    &+ L^2[ \beta_4^{n_c} + (n_g - 1)\delta_2(n_c, n_g)] \left[L n_g + (n_g - 1)\right] \left[(n_g-1) \left( \beta_1^{n_c} + \delta_1(n_c, n_g)\right) + \beta_2^{n_c} \right] \\
    b_2  = & \mu n_g \beta_4^{n_c}L \left((n_g-1)(\beta_1^{n_c} + \delta_1(n_c, n_g)) + \beta_2^{n_c} \right), \text{ and }\; b_3 = \tfrac{\mu n_g}{2} \left( \tfrac{1 - \beta_1^{n_c}}{4}\right) \left(\tfrac{1 - \beta_3^{n_c}}{4}\right)
\end{align*}
and $\delta_1(n_c, n_g)$ and $\delta_2(n_c, n_g)$ are defined in \eqref{eq:delta_errors}, 
then, for all $\epsilon > 0$ there exists a constant $C_{\epsilon} > 0$ such that, for all $k\geq 0$,
\begin{align*}
    \|r_{k}\|_2 \leq C_{\epsilon}(\rho(B({n_c, n_g})) + \epsilon)^k \|r_0\|_2, \quad \text{where } \; \rho(B({n_c, n_g})) < 1.
\end{align*} 
\etheorem
\bproof
By the binomial expansion of $(1-\alpha\mu)^{n_g}$ and the condition that $\alpha \leq \tfrac{1}{L n_g}$, it follows that $1 - \alpha \mu n_g \leq (1 - \alpha \mu)^{n_g} \leq 1 - \alpha \mu n_g + \alpha^2 \mu^2 \tfrac{n_g (n_g-1)}{2} $. Following a similar approach to \cite[Theorem 2]{pu2020push}, since the step size satisfies \eqref{eq : alpha mult grads gen}, the first, second and third diagonal terms of $B(n_c, n_g)$
can be upper bounded as
\begin{align*}
    (1 - \alpha \mu)^{n_g} + \alpha^2 L^2 n_g (n_g - 1)  \leq 1 - \alpha \mu n_g + \alpha^2 (L^2 + \tfrac{\mu^2}{2}) n_g (n_g - 1)& < 1 - \tfrac{\alpha \mu n_g}{2}, \\ 
    \beta_1^{n_c} + \alpha^2 L^2 (n_g - 1) \delta_1(n_c, n_g) & < \tfrac{3 + \beta_1^{n_c}}{4}, \\
    \beta_3^{n_c} + \alpha \beta_4^{n_c} L + \alpha L (n_g - 1) \delta_2(n_c, n_g) & < \tfrac{3 + \beta_3^{n_c}}{4}.
\end{align*}
With the above bounds, $(1-\alpha \mu)^{n_g} \geq 1 - \alpha\mu n_g$ and $\|\Zmbf_1^{n_c} - I_{nd}\| \leq 2$, we construct the $3\times 3$ matrix $\Tilde{B}(n_c, n_g)$ that has entries $\tilde{b}_{ij}$ defined as follows:
\begin{align*}
    &\Tilde{b}_{11} = 1 - \tfrac{\alpha \mu n_g}{2}, \quad \Tilde{b}_{12} = \tfrac{\alpha L n_g}{\sqrt{n}} \left(1 + \alpha L (n_g - 1)\right),\quad \Tilde{b}_{13} = \tfrac{\alpha^2 L n_g (n_g - 1)}{\sqrt{n}}\\
    &\Tilde{b}_{21} = \sqrt{n}\alpha^2 L^2 (n_g - 1) \delta_1(n_c, n_g), \quad \Tilde{b}_{22} = \tfrac{3 + \beta_1^{n_c}}{4},\\
    & \Tilde{b}_{23} = \alpha\left((n_g-1) (\beta_1^{n_c}  + \alpha L \delta_1(n_c, n_g)) + \beta_2^{n_c}\right), \\
    &\Tilde{b}_{31} = \sqrt{n}\alpha L^2 \left(\beta_4^{n_c} + (n_g - 1)\delta_2(n_c, n_g)\right), \\
    &\Tilde{b}_{32} = \beta_4^{n_c}L(2 + \alpha L) + \alpha L^2 (n_g - 1) \delta_2(n_c, n_g), \quad \Tilde{b}_{33} = \tfrac{3 + \beta_3^{n_c}}{4},
\end{align*}
such that $0 \leq B(n_c, n_g) \leq \Tilde{B}(n_c, n_g)$ and by \cite[Corollary 8.1.19]{horn2012matrix}, $\rho(B(n_c, n_g)) \leq \rho(\Tilde{B}(n_c, n_g))$. Following \cite[Lemma 5]{pu2021distributed} derived from the Perron-Forbenius Theorem \cite[Theorem 8.4.4]{horn2012matrix} for a $3\times3$ matrix, when the matrix $\Tilde{B}(n_c, n_g)$ is nonnegative and irreducible, it is sufficient to show that the diagonal elements of $\Tilde{B}(n_c, n_g)$ are less than one and $\det(I_3 - \Tilde{B}(n_c, n_g)) > 0$ in order to guarantee $\rho(\Tilde{B}(n_c, n_g)) < 1$ which suffices to show $\rho(B(n_c, n_g)) < 1$. 

Consider the diagonal elements of the matrix $\Tilde{B}(n_c, n_g)$. The first element is $1 - \frac{\alpha \mu n_g}{2} \leq 1 - \frac{\mu}{2L} < 1$ by \eqref{eq : alpha mult grads gen}. The second element is $\frac{3 + \beta_1^{n_c}}{4} < 1$ as $\beta_1 < 1$. Finally the third element is $\frac{3 + \beta_3^{n_c}}{4} < 1$ as $\beta_3 < 1$. Next, let us consider,
\begin{align*}
    &\det(I_3 - \Tilde{B}(n_c, n_g)) \\
    =& \tfrac{\alpha \mu n_g}{2} \left( \tfrac{1 - \beta_1^{n_c}}{4}\right) \left(\tfrac{1 - \beta_3^{n_c}}{4}\right) - \alpha^3 L^3 n_g (n_g - 1)\left[\beta_4^{n_c}  \left(\tfrac{1 - \beta_1^{n_c}}{4}\right) + \delta_1(n_c, n_g)  \left(\tfrac{1 - \beta_3^{n_c}}{4}\right)\right]\\
    &- \tfrac{\alpha^2 \mu L n_g}{2} \left[(n_g-1)(\beta_1^{n_c}+ \alpha L \delta_1(n_c, n_g)) + \beta_2^{n_c} \right]\left[\beta_4^{n_c}(2 + \alpha L) + \alpha L (n_g - 1)\delta_2(n_c, n_g)\right] \\
    &-\alpha^4 L^4 n_g (n_g - 1)^2\delta_1(n_c, n_g) \left[2\beta_4^{n_c} + \alpha L \left( \beta_4^{n_c} + (n_g - 1)\delta_2(n_c, n_g)\right) \right] \\
    & - \alpha^3 L^3 n_g\left(1 + \alpha (n_g - 1)\right)\left[ \beta_4^{n_c}  +(n_g - 1)\delta_2(n_c, n_g)\right] \left[(n_g-1) (\beta_1^{n_c} + \alpha L \delta_1(n_c, n_g)) + \beta_2^{n_c} \right] \\
    & - \alpha^3 L^3 n_g (n_g - 1)^2\left[ \delta_2(n_c, n_g)  \left(\tfrac{1 - \beta_1^{n_c}}{4}\right) + \alpha \delta_1(n_c, n_g)  \left(\tfrac{1 - \beta_3^{n_c}}{4}\right)\right]\\  
    \geq& \alpha(- b_1 \alpha^2 - b_2 \alpha + b_3) = -b_1 \alpha(\alpha - \alpha_l)(\alpha - \alpha_u)
\end{align*}
where the inequality is due to $\alpha L n_g \leq 1$ and thus $\alpha L \leq 1$ as $n_g \geq 1$, and 
\begin{align*}
\alpha_l = \tfrac{-b_2 - \sqrt{b_2^2 + 4b_1 b_3}}{2b_1} \quad \text{and} \quad 
\alpha_u = \tfrac{-b_2 + \sqrt{b_2^2 + 4b_1 b_3}}{2b_1}. 
\end{align*}
Observe that $\alpha_l < 0 < \alpha_u$ since $b_1, b_2, b_3 \geq 0$. From \eqref{eq : alpha mult grads gen}, we have $0<\alpha < \alpha_u$. Therefore, $\det(I_3 - \Tilde{B}({n_c, n_g})) > 0$, which combined with the fact that the diagonal elements of the matrix are less than 1, implies $\rho(B(n_c, n_g)) \leq \rho(\Tilde{B}(n_c, n_g)) < 1$.

Finally, we bound the norm of error vector $\|r_k\|_2$ by telescoping $r_{i+1} \leq B(n_c, n_g) r_{i}$ from $i = 0$ to $k-1$ and triangle inequality as
\begin{align*}
    \|r_{k}\|_2 &\leq \|B(n_c, n_g)^k\|_2\|r_0\|_2.
\end{align*}
From \cite[Corollary 5.6.13]{horn2012matrix}, we can bound $\|B(n_c, n_g)^k\|_2 \leq C_{\epsilon}(\rho(B(n_c, n_g)) + \epsilon)^k$ where $\epsilon > 0$ and $C_{\epsilon}$ is a positive constant depending on $B(n_c, n_g)$ and $\epsilon$.
\eproof
\qed

\bremark 
Linear convergence can be established for the iterates generated by \cref{alg : Deterministic} by treating inner iterations as a special case of of time-varying networks and following the analysis techniques in \cite{nedic2017geometrically, nedic2017achieving, nguyen2022performance}. Such analysis would ensure descent in each inner iteration and require that the step size satisfy a pessimistic $\mathcal{O}(\frac{1}{n_g^2})$ condition. Our analysis takes a different approach; we quantify the error across outer iterations, and as a result, the condition on the step size is less pessimistic, i.e., $\mathcal{O}(\frac{1}{n_g})$. That said, the analysis involves a complicated error recursion that makes it difficult to explicitly quantify the rate of convergence of \cref{alg : Deterministic}.
\eremark

Similar to \cref{th. general g=1 rate bound}, the only constraint \cref{th. alpha bound g > 1} imposes on the system is $\beta_1, \beta_3 < 1$. This implies the communication matrices $\Wmbf_1$ and $\Wmbf_3$ must represent connected networks (not necessarily the same network) even when multiple communication and multiple computation steps are performed. \cref{th. alpha bound g > 1} does not impose any restrictions on the relation among $\Wmbf_1$, $\Wmbf_2$, $\Wmbf_3$ and $\Wmbf_4$. Thus, it allows for more flexibility than the structures considered in the literature even when multiple communication and multiple computation steps are performed.  
\cref{th. alpha bound g > 1} uses a relaxation of the original matrix $B(n_c, n_g)$ to provide a more pessimistic step size condition than required. But observe, when $n_g = 1$, \eqref{eq : alpha mult grads gen} recovers the $\mathcal{O}(L^{-1}\kappa^{-0.5})$ step size condition of \cref{th. general g=1 step cond}, suggesting it might not be very pessimistic.

Based on \cref{th. alpha bound g > 1}, the step size conditions for methods described in \cref{tab: Algorithm Def} can be derived. We omit these conditions as they are complex and do not offer any additional insights. 
We also omit the counterpart to \cref{th. general g=1 rate bound} as the matrix $B(n_c, n_g)$ is now a dense matrix, thus any such bounds are again highly complex and do not offer strong insights into the effects of communication and computation on the convergence rate. If $B(n_c, n_g)$ is a reducible matrix, the analysis for the progression of $r_k$ can be further simplified from \cref{lem:lyapunov g > 1}. The analysis for this case is presented in \cref{sec.full graph res} with the examples of \texttt{GTA-2} and \texttt{GTA-3} when $\Wmbf = \frac{1_n1_n^T}{n}$, i.e., $\beta= 0$.


\subsection{Fully connected network} \label{sec.full graph res}
In this section, we analyze the methods defined in \cref{tab: Algorithm Def} under a fully connected network. While showing linear convergence of \texttt{GTA} in \cref{th. alpha bound g > 1}, we assume $B(n_c, n_g)$ is an irreducible matrix. When the network is fully connected, i.e., $\Wmbf = \tfrac{1_n1_n^T}{n}$ and $\beta = 0$, the assumption does not hold for \texttt{GTA-2} and \texttt{GTA-3} as the matrices  $B_2(n_c, n_g)$ and $B_3(n_c, n_g)$ defined by \cref{col. B special cases} are reducible. For \texttt{GTA-1}, such an issue does not arise as $A_1(n_c, n_g)$ defined in \cref{col. B special cases} is irreducible for all $\beta \in [0,1]$. Thus, we now present sufficient conditions for linear rate of convergence and the convergence rate for \texttt{GTA-2} and \texttt{GTA-3} for the special case of fully connected networks.

\btheorem \label{th. fully connected}
Suppose \cref{asum.convex and smooth} holds, $\Wmbf = \frac{1_n1_n^T}{n}$ and a finite number of computation steps are performed each outer iteration of \texttt{GTA-3} defined in \cref{tab: Algorithm Def} (i.e., $1 \leq n_g < \infty$). If $\alpha < \min \left\{\frac{\mu}{(2L^2 + \mu^2)(n_g-1)}, \frac{1}{L n_g}\right\}$, then for all $k \geq 0$,
\begin{align*}
    \|\xbar_{k+1, 1} - x^*\|_2 & \leq \left(
    (1 - \alpha \mu)^{n_g} + \alpha^2 L^2 n_g(n_g - 1) \right)  \|\xbar_{k, 1} - x^*\|_2.
\end{align*}
Moreover, suppose the number of computation steps performed each outer iteration of \texttt{GTA-2} and \texttt{GTA-3} defined in \cref{tab: Algorithm Def} is set to one (i.e., $n_g =1$). If $\alpha \leq \frac{1}{L}$, then for both the methods, for all $k \geq 0$,
\begin{align*}
    \|\xbar_{k+1, 1} - x^*\|_2 & \leq 
    (1 - \alpha \mu) \|\xbar_{k, 1} - x^*\|_2.
\end{align*}
\etheorem
\bproof
When we substitute $\beta = 0$ in \cref{col. B special cases} as $\alpha < \frac{1}{n_g L}$, the matrices $B_2(n_c, n_g)$ and $B_3(n_c, n_g)$ now have rows of zeros that make them reducible. Thus, we reduce these matrices by ignoring the error terms corresponding to the row of zeros. This yields the following systems for the progression of errors in these methods,
\begin{align}
    &\mbox{\texttt{GTA-2: }}
     \Tilde{r}_{k+1} \label{eq : g > 1 reduced 2}    
    \leq \begin{bmatrix}
    (1 - \alpha \mu)^{n_g} + \alpha^2 L^2 n_g(n_g - 1) & \frac{\alpha^2 L n_g(n_g - 1)}{\sqrt{n}}\\
    \sqrt{n}\alpha L^2 \Tilde{\delta}(n_c, n_g) & \alpha L\Tilde{\delta}(n_c, n_g)\\
    \end{bmatrix} \Tilde{r}_{k}, \\
    &\mbox{\texttt{GTA-3: }}
        \|\xbar_{k+1, 1} - x^*\|_2  \leq \left(
    (1 - \alpha \mu)^{n_g} + \alpha^2 L^2 n_g(n_g - 1) \right)
        \|\xbar_{k, 1} - x^*\|_2, \label{eq : g > 1 reduced 3} 
\end{align}
where $\Tilde{\delta}(n_c, n_g) = 1 + 2(n_g - 1)\left(2 + \tfrac{1}{n_g}\right)$ and $\Tilde{r}_{k} = \begin{bmatrix}
        \|\xbar_{k, 1} - x^*\|_2\\
        \|\ymbf_{k, 1} - \Bar{\ymbf}_{k, 1}\|_2\\
    \end{bmatrix}$. \\
    By $\alpha < \frac{\mu}{(2L^2 + \mu^2)(n_g-1)}$ and $(1-\alpha\mu)^{n_g} \leq 1 - \alpha\mu n_g + \alpha^2\mu^2 \tfrac{n_g(n_g-1)}{2}$ from \cref{th. alpha bound g > 1},
\begin{align*}
    (1 - \alpha \mu)^{n_g} + \alpha^2 L^2 n_g(n_g - 1) \leq 1 - \alpha \mu n_g + \alpha^2 \left(L^2 + \tfrac{\mu^2}{2}\right) n_g(n_g - 1) < 1,
\end{align*}
and thus the result for \texttt{GTA-3} follows. When the number of computation steps performed each outer iteration is set to one, i.e., $n_g = 1$, 
the result for \texttt{GTA-3} follows by substituting $n_g=1$ in \eqref{eq : g > 1 reduced 3}, where $1 - \alpha\mu < 1$ as $\alpha \leq \frac{1}{L}$. Substituting $n_g=1$ in \eqref{eq : g > 1 reduced 2} for \texttt{GTA-2} yields,  $\Tilde{r}_{k+1} \leq \begin{bmatrix}
    1 - \alpha \mu & 0\\
    \sqrt{n}\alpha L^2 & \alpha L\\
    \end{bmatrix}
        \Tilde{r}_{k}$, 
where the bound on optimization error is independent of the consensus error in $\ymbf_{k, 1}$.
Thus, we obtain $\|\xbar_{k+1, 1} - x^*\|_2 \leq \left(1 - \alpha \mu \right) \|\xbar_{k, 1} - x^*\|_2$ for \texttt{GTA-2}.
\eproof
\qed

By \cref{th. fully connected}  if the network is fully connected and a single computation step is performed, i.e., $n_g = 1$, \texttt{GTA-2} and \texttt{GTA-3} display gradient descent performance. For \texttt{GTA-2}, when the network is fully connected and $n_g > 1$, the convergence rate can be expressed as the spectral radius of the $2\times2$ matrix in \eqref{eq : g > 1 reduced 2}.


\section{Numerical Experiments}\label{sec.num_exp}

\setcounter{footnote}{1} 
In this section, we 
illustrate the empirical performance of the methods defined in \cref{tab: Algorithm Def} using Python implementations\footnote{Our code will be made publicly available upon publication of the manuscript. Github repository: \url{https://github.com/Shagun-G/Gradient-Tracking-Algorithmic-Framework}. Moreover, additional extensive numerical results can be found in the same repository.}. 
The aim of this section is to show, over multiple problems, that different communication strategies and the balance between communication and computation steps 
can substantially effect the algorithm's performance. Specifically, we establish the relative performance of the methods defined in \cref{tab: Algorithm Def} and illustrate the benefits of the flexibility in terms of communication and computation steps.

We present results on two problems: $(1)$ a synthetic strongly convex quadratic problem (\cref{sec : quads}); and, $(2)$ binary classification logistic regression problems over the mushroom and australian datasets \cite{Dua:2019} (\cref{sec : logistic}). We 
investigated 
two  
network structures (different mixing matrix $\Wmbf$) with $n = 16$ nodes: 
$(1)$ a connected cyclic network ($\beta = 0.992$) where all nodes have two neighbours; and, $(2)$ a connected star network ($\beta = 0.95$) where all nodes are connected to a single central node. 
Both networks have low connectivity (i.e., high $\beta$). We should note that the performance of \cref{alg : Deterministic} with multiple communication steps is equivalent to the performance over a network with higher connectivity (i.e., lower $\beta$).

The methods defined in \cref{tab: Algorithm Def} 
are denoted 
as \texttt{GTA}$-i(n_c, n_g)$, $i = 1, 2, 3$, where $n_c$ and $n_g$ are the number of communication and computation steps, respectively. We 
tested 5 values of $n_c$ and $n_g$ for each of the methods; 
$n_c \in \{1, 5, 10, 50, 100\}$ and $n_g \in \{ 1, 5, 20, 50, 100\}$. We compared the  performance of 
popular gradient tracking methods, which are special cases of our generalized framework. 
The step sizes were tuned over the set $\{2^{-t} | t = 0, 1, 2, .., 20\}$ for all algorithms and problems, and the initial iterates for all algorithms, problems and nodes were set to the zero vector (i.e., $\xmbf_k = \mathbf{0}$). 
The performance of the methods was measured 
in terms of the optimization error ($\|\bar{x}_k - x^*\|_2$) and the consensus error ($\|\xmbf_k - \xbb_k\|_2$). We do not report 
the consensus error in the auxiliary variable $\ymbf_k$ ($\|\ymbf_k - \ybb_k\|_2$) as this measure does not provide any significant additional insights about the performance of the algorithms.  
The optimal solution $x^*$ for quadratic problem was obtained analytically and for the logistic regression problems was obtained by
running gradient descent in the centralized setting to high accuracy, i.e., $\|\nabla f(x^*)\|_2 \leq 10^{-12}$. 

\subsection{Quadratic Problems}\label{sec : quads}

We first consider quadratic problems
\begin{align*} 
    f(x) = \frac{1}{n} \sum_{i=1}^n \frac{1}{2}x^TQ_ix + b_i^Tx,
\end{align*}
where $Q_i \in \mathbb{R}^{10 \times 10}$, $Q_i \succ 0$ and $b_i \in \mathbb{R}^{10}$ is the local information at each node $i \in \{1, 2, .., n\}$, and $n=16$. Each local problem is strongly convex and was generated using the procedure described in \cite{mokhtari2016network}, with global condition number $\kappa \approx 10^4$. 

\cref{fig : Quadratic cyclic,fig : Quadratic Star} show the performance of \texttt{GTA-1}, \texttt{GTA-2} and \texttt{GTA-3} over a cyclic network and a star network, respectively. Our first observation, from the iteration plots in both the figures, is that the optimization error and consensus error converge at a linear rate for all methods, matching the theoretical results of \cref{sec.theory}. Moreover, improvements in the rates of convergence of all methods are observed as a result of the flexibility in terms of the number of communication and computation steps. Specifically, the consensus error is improved (and on par optimization error) when multiple communication steps with single computation step are performed (see \texttt{GTA-i}($1$, 1) vs. \texttt{GTA-i}($n_c$, 1) lines), and the optimization error is improved (and on par consensus error) when multiple computation steps with same number of communication steps are performed (see \texttt{GTA-i}($n_c$, 1) vs. \texttt{GTA-i}($n_c$, $n_g$) lines). 
These observations match the theory presented in Section~\ref{sec.mult grads}. 
That being said, 
these improvements come at a higher cost in terms of total communication or computation steps, respectively, and an optimal choice of $(n_c, n_g)$ depends on the exact cost structure that combines the complexity of both these steps; see e.g., \cite{berahas2018balancing}. Finally, we also observe that \texttt{GTA-2} and \texttt{GTA-3} outperform \texttt{GTA-1} in terms of optimization error and achieve similar consensus error. The performance of \texttt{GTA-2} and \texttt{GTA-3} is very similar for this problem, we suspect the reason for this behavior is due to the large $\beta$ and the high condition number ($\kappa \approx 10^4$) that dominate the rate constant; see \cref{col. g=1 rate bound}. 

\begin{figure}[h]
\centering
\includegraphics[width=\textwidth]{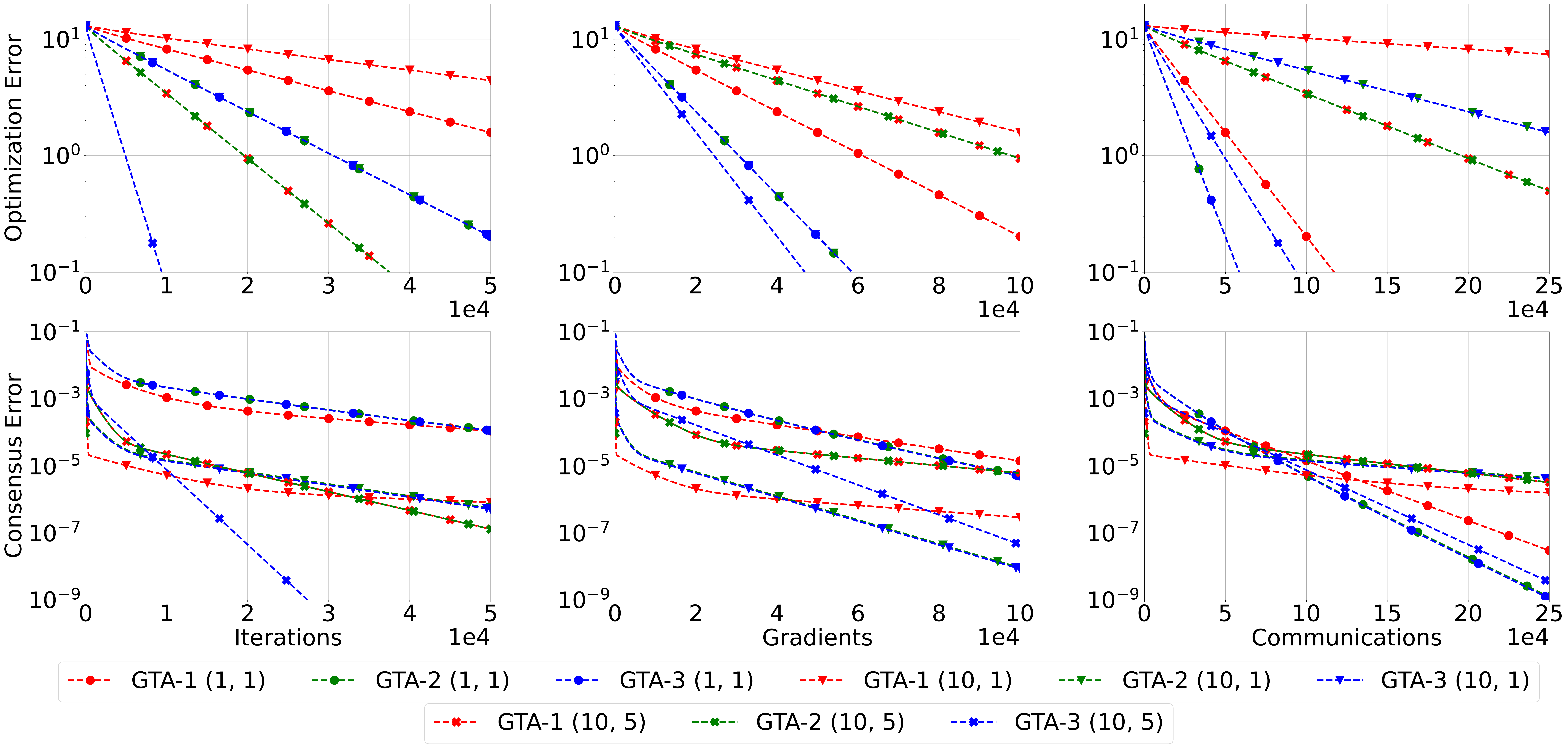}  
\caption{Optimization Error ($\|\xbar_k - x^*\|_2$) and Consensus Error ($\|\xmbf_k - \xbb_k\|_2$) of \texttt{GTA-1}, \texttt{GTA-2} and \texttt{GTA-3} with respect to number of iterations, communications and gradient evaluations for a synthetic quadratic problem ($n = 16$, $d = 10$, $\kappa = 10^4$) over a cyclic network ($\beta =  0.992$).}
\label{fig : Quadratic cyclic}
\end{figure}

\begin{figure}[h]
\centering
\includegraphics[width=\textwidth]{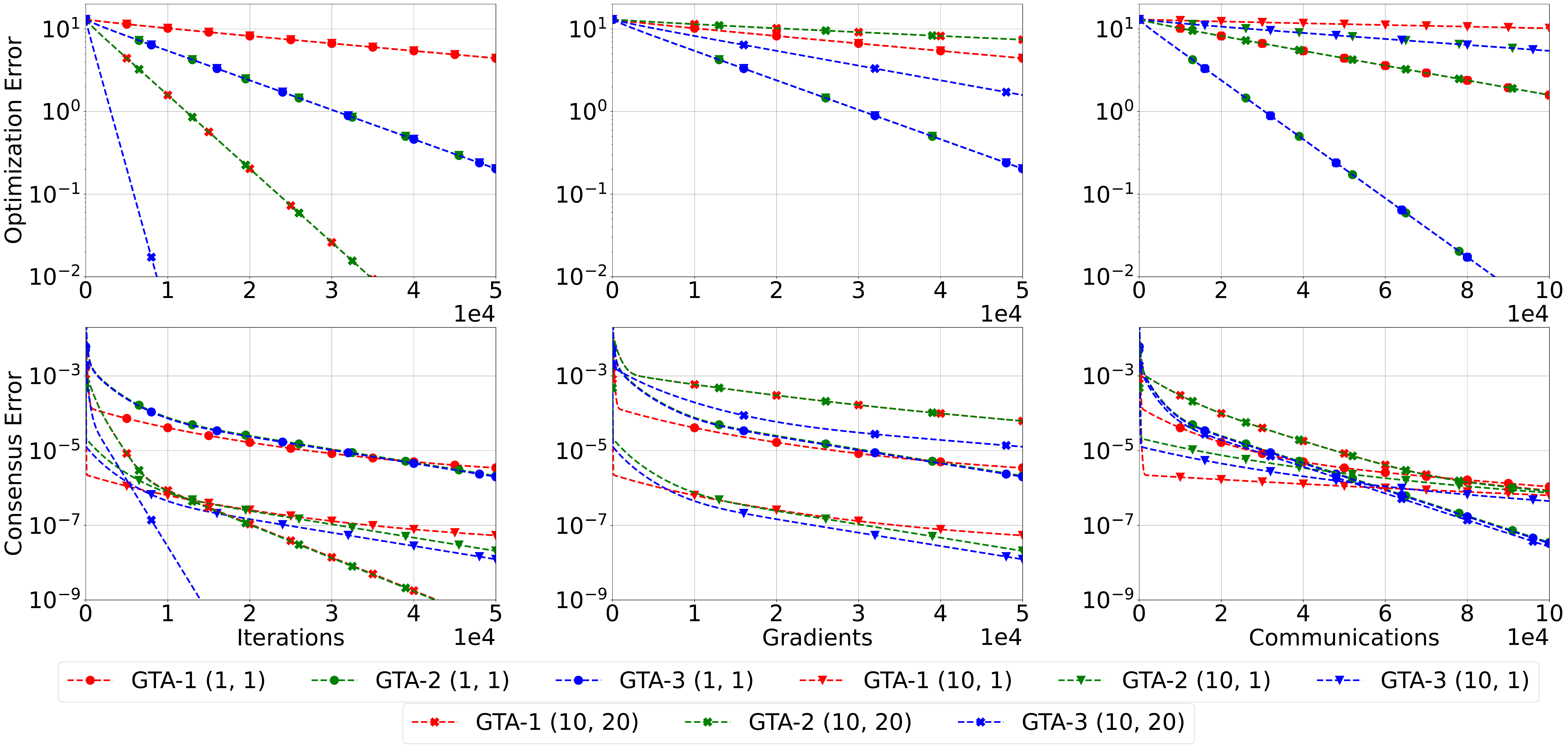}    
\caption{Optimization Error ($\|\xbar_k - x^*\|_2$) and Consensus Error ($\|\xmbf_k - \xbb_k\|_2$) of \texttt{GTA-1}, \texttt{GTA-2} and \texttt{GTA-3} with respect to number of iterations, communications and gradient evaluations for a synthetic quadratic problem ($n = 16$, $d = 10$, $\kappa = 10^4$) over star network($\beta =  0.95$).}
\label{fig : Quadratic Star}
\end{figure}

\subsection{Binary Classification Logistic Regression}\label{sec : logistic}
Next, we consider $\ell_2$-regularized binary classification logistic regression problems of the form
\begin{align*} 
    f(x) &= \frac{1}{n} \sum_{i=1}^n \frac{1}{n_i}\log(1 + e^{-b_i^TA_ix}) + \frac{1}{n_i}\|x\|_2^2, 
\end{align*}
where each node $i \in \{1, 2, .., n\}$ has a portion of data samples $A_i \in \mathbb{R}^{n_i \times d}$ and corresponding labels $b_i \in \{0, 1\}^{n_i}$. Experiments were performed over the mushroom dataset ($n = 16$, $d = 117$, $\sum_{i=1}^n n_i = 8124$) and the australian dataset ($n = 16$, $d = 41$, $\sum_{i=1}^n n_i = 690$) \cite{Dua:2019}. 

\cref{fig : Mushroom Cyclic,fig : Austrailian Star} show the performance of \texttt{GTA-1}, \texttt{GTA-2} and \texttt{GTA-3} over a cyclic network ($\beta =  0.992$) for the mushroom dataset and a star network for the australian dataset ($\beta =  0.95$), respectively. Similar observations  to those made for the quadratic problem with respect to the effect of performing multiple communication and computation steps can also be made for these problems. 
Additionally, we observe that \texttt{GTA-3} outperforms \texttt{GTA-2} on these problems. 
We should note that although \texttt{GTA-3} performs the best within these experiments, it also brings certain implementation constraints; see \cref{sec.methods}.

\begin{figure}[]
\centering
\includegraphics[width=\textwidth]{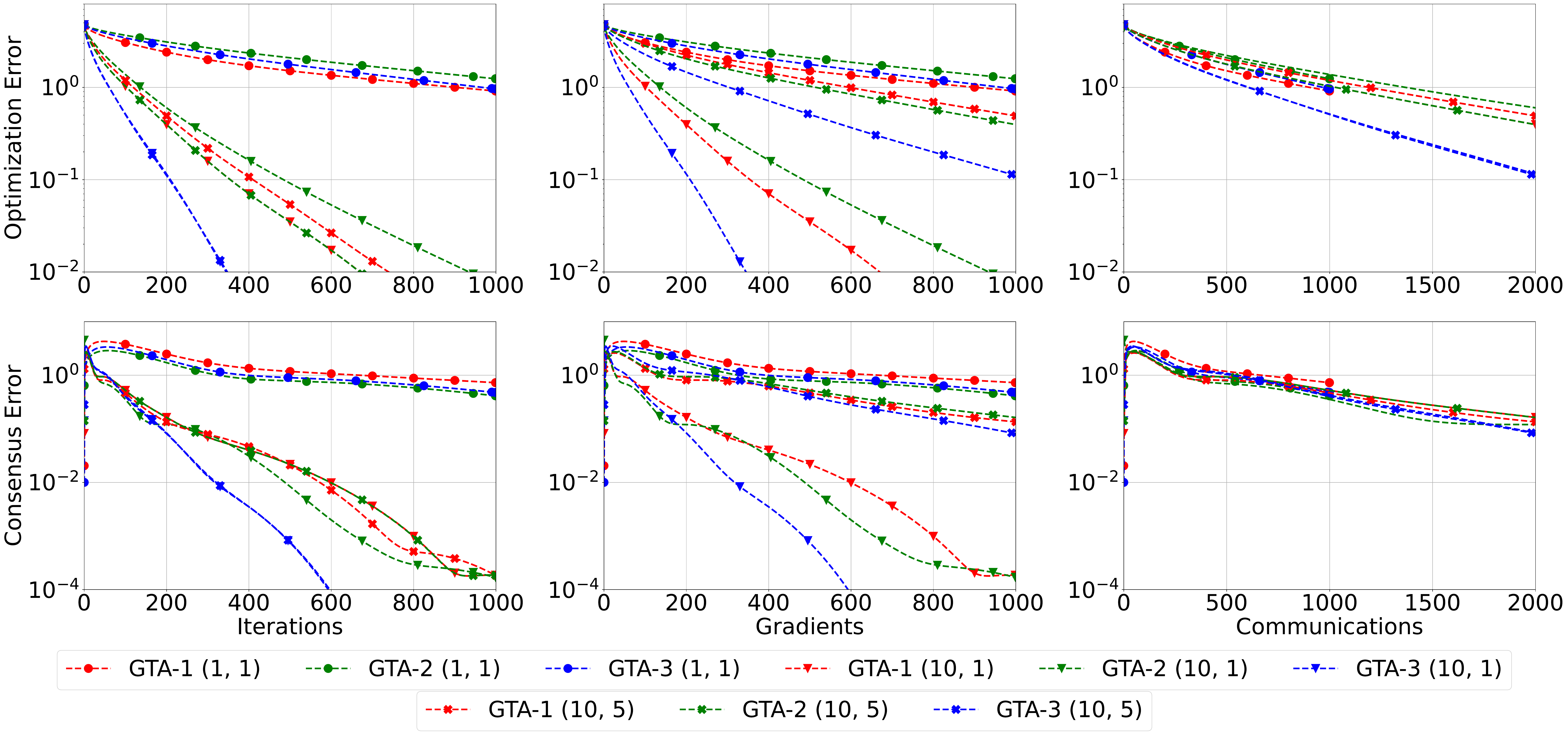} 
\caption{Optimization Error ($\|\xbar_k - x^*\|_2$) and Consensus Error ($\|\xmbf_k - \xbb_k\|_2$) of \texttt{GTA-1}, \texttt{GTA-2} and \texttt{GTA-3} with respect to number of iterations, communications and gradient evaluations for binary logistic regression on Mushroom dataset ($n = 16$, $d = 117$, $\sum_{i=1}^n n_i = 8124$) over cyclic network ($\beta =  0.992$).}
\label{fig : Mushroom Cyclic}
\end{figure}

\begin{figure}[]
\centering
\includegraphics[width=\textwidth]{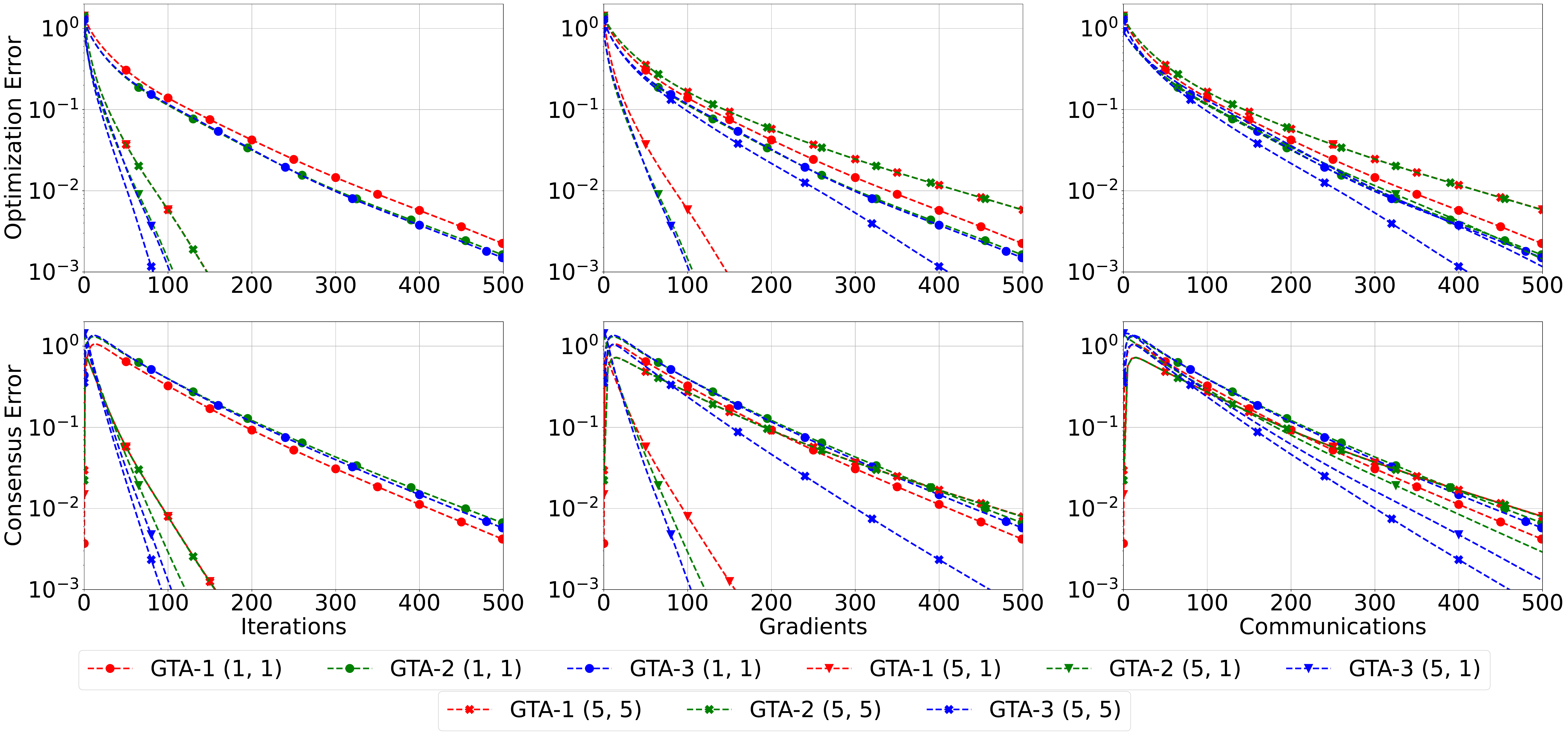}  
\caption{Optimization Error ($\|\xbar_k - x^*\|_2$) and Consensus Error ($\|\xmbf_k - \xbb_k\|_2$) of \texttt{GTA-1}, \texttt{GTA-2} and \texttt{GTA-3} with respect to number of iterations, communications and gradient evaluations for binary logistic regression on Australian dataset ($n = 16$, $d = 41$, $\sum_{i=1}^n n_i = 690$) over star network ($\beta = 0.95$).
}
\label{fig : Austrailian Star}
\end{figure}
  

\section{Final Remarks}\label{sec.conc}

In this paper, we have proposed a framework that unifies and generalizes communication strategies in gradient tracking methods with flexibility in the number of communication and computation steps performed at every iteration. We have established convergence guarantees for the proposed gradient tracking framework. Specifically, we have shown linear convergence for the general framework and the special cases of gradient tracking methods. Moreover, we have shown the positive influence of performing multiple communication steps at every iteration on the convergence rate and provide results that allow for the direct comparison of popular gradient tracking methods. Our experiments on quadratic and logistic regression problems illustrate the effects of different communication strategies and the benefits of the flexibility in terms of iterations and number of communication and computation steps. The advantages of the proposed framework can be further realized when the actual cost, i.e., a combination of the complexity of both communication and computation steps that is application specific, is considered. 
Finally, the algorithmic framework can be extended to other interesting settings such as nonconvex problems, stochastic local information, asynchronous updates, and higher-order approaches.

\newpage
\bibliographystyle{plain}
{\small
\bibliography{references}}


\end{document}